\documentclass[12pt]{amsart}
\usepackage{amssymb}
\usepackage{amsmath}
\usepackage{amsthm}
\usepackage{amsfonts}
\usepackage{bbm}
\usepackage{todonotes} 
\usepackage{tikz}
\usetikzlibrary{arrows}
\usepackage{graphicx}

\usepackage{breakcites}
\usepackage{color}
\usepackage{graphicx} 
\numberwithin{equation}{section}
\usepackage{breakcites}
\usepackage{color}
\usepackage{graphicx}
\usepackage{epstopdf}
\usepackage{tikz}
\usepackage{hyperref}

\newtheorem{thm}{Theorem}[section]
\newtheorem{lemma}[thm]{Lemma}
\newtheorem{cor}[thm]{Corollary}
\newtheorem{prop}[thm]{Proposition}
\newtheorem{conj}[thm]{Conjecture}

\theoremstyle{definition}
\newtheorem{definition}[thm]{Definition}
\newtheorem{example}[thm]{Example}

\theoremstyle{remark}
\newtheorem{remark}[thm]{Remark}

\renewcommand{\S}{\mathfrak S}

\newcommand\Q{{\mathbb Q}}
\newcommand\Z{{\mathbb{Z}}}

\newcommand\PP{{\mathbb{P}}}
\newcommand\N{{\mathbb{N}}}

\DeclareMathOperator{\rank}{rank}
\DeclareMathOperator{\inv}{inv}
\DeclareMathOperator{\DES}{DES}
\DeclareMathOperator{\sgn}{sgn}
\DeclareMathOperator{\exc}{exc}
\DeclareMathOperator{\asc}{asc}
\DeclareMathOperator{\ASC}{ASC}
\DeclareMathOperator{\FIX}{FIX}
\DeclareMathOperator{\dtr}{det}
\DeclareMathOperator{\des}{des}

\begin{document}

\title[Chromatic quasisymmetric functions of digraphs]{A directed graph generalization of chromatic quasisymmetric functions}
\author[B. Ellzey]{Brittney Ellzey}
\address{Department of Mathematics, University of Miami, Coral Gables, FL 33146}
\email{bellzey@math.miami.edu}
\thanks{$^{1}$Supported in part by NSF Grant 1202755}

\begin{abstract} Stanley defined the chromatic symmetric function of a graph, and Shareshian and Wachs introduced a refinement, namely the chromatic quasisymmetric function of a labeled graph.  In this paper, we define the chromatic quasisymmetric function of a \textit{directed} graph, which agrees with the Shareshian-Wachs definition in the acyclic case.  We give an F-basis expansion for all digraphs in terms of a permutation statistic, which we call G-descents.  We use this expansion to derive a p-positivity formula for all digraphs with symmetric chromatic quasisymmetric functions.  We show that the chromatic quasisymmetric functions of a certain class of digraphs, called circular indifference digraphs, have symmetric coefficients.  We present an e-positivity formula for the chromatic quasisymmetric function of the directed cycle, which is a t-analog of a result of Stanley.   Lastly, we give a generalization of the Shareshian-Wachs e-positivity conjecture to a larger class of digraphs.
\end{abstract}

\date{September 1, 2017}

\vspace*{-.2in} \maketitle

\tableofcontents

\makeatletter
\providecommand\@dotsep{5}
\makeatother

\newpage

\section{Introduction}

Let $G=(V,E)$ be a (simple) graph.  A proper coloring, $\kappa: V \rightarrow \PP$, of $G$ is an assignment of positive integers, which we can think of as colors, to the vertices of $G$ such that adjacent vertices have different colors; in other words, if $\{i,j\} \in E,$ then $\kappa(i) \neq \kappa(j).$  The chromatic polynomial of $G$, denoted $\chi_G(k),$ gives the number of proper colorings of $G$ using $k$ colors.  Stanley \cite{CSF} defined a symmetric function refinement of the chromatic polynomial called the \textit{chromatic symmetric function} of a graph.  For any graph $G,$ let $\kappa(G)$ denote the set of proper colorings of $G.$ If we let $V=\{v_1, v_2,\cdots v_n\},$ then the chromatic symmetric function of $G$ is defined as $$X_G({\bf x}) = \displaystyle \sum_{\kappa \in \kappa(G)} {\bf x}_\kappa, $$ where ${\bf x}_\kappa = x_{\kappa(v_1)}x_{\kappa(v_2)}\cdots x_{\kappa(v_n)}$.  For any formal power series, $f({\bf x}),$ with ${\bf x} = x_1, x_2, \cdots,$ we let $f(1^k)$ denote the specialization of $f({\bf x})$ obtained by setting $x_i = 1$ for $i \leq k$ and $x_j = 0$ for $j > k.$ It is easy to see that the chromatic symmetric function of a graph refines the chromatic polynomial, because $X_G(1^k) = \chi_G(k).$  

The chromatic symmetric function of a graph is symmetric, because permuting the variables simply corresponds to relabeling the colors.  Hence, for any graph G, $X_G({\bf x}) \in \Lambda_\Q,$ where $\Lambda_\Q$ is the $\Q$-algebra of symmetric functions in the variables ${\bf x} = x_1, x_2,\cdots$ with coefficients in $\Q.$  For $n \in \N$, let $\Lambda_\Q^n$ denote the $\Q$-vector space of homogeneous symmetric functions of degree $n$, and let $\lambda \vdash n$ mean that $\lambda$ is an integer partition of $n$.  For any basis, $b = \{b_\lambda \mid \lambda \vdash n\},$ of $\Lambda_\Q^n,$ we say that a symmetric function, $f \in \Lambda_\Q^n$ is \textit{b-positive} if the expansion of the symmetric function in terms of the $b_\lambda$ basis has nonnegative coefficients.  The symmetric function bases we focus on in this paper are the power sum symmetric function basis, $p = \{p_{\lambda} \mid \lambda \vdash n\},$ and the elementary symmetric function basis, $e = \{e_{\lambda} \mid \lambda \vdash n\}.$  Throughout this paper, we use $\omega$ to denote the usual involution on $\Lambda_\Q$ sending the homogeneous symmetric function to the elementary symmetric function. We assume the reader is familiar with basic properties of symmetric and quasisymmetric functions, which can be found in \cite{EC2}. 

One of the most well-known conjectures involving chromatic symmetric functions of graphs is about their $e$-positivity.  Recall that a poset is \textit{(a+b)-free} if it has no induced poset that is the union of a chain with $a$ elements and a chain $b$ elements.  The \textit{incomparability graph} of a poset $P,$ denoted $inc(P),$ is the graph with the elements of $P$ as vertices and edges between incomparable elements of $P$.

\begin{conj} [Stanley-Stembridge \cite{Stem} \cite{CSF}]
Let P be a $(3+1)$-free poset.  Then $X_{inc(P)}({\bf x})$ is $e$-positive.
\end{conj}

Chromatic symmetric functions have been extensively studied.  Some of these studies include \cite{GP}, \cite{StanGar}, \cite{Gash2}, \cite{Gash}, \cite{Chow}, \cite{Chow2}, \cite{MMW}, \cite{Hump}, \cite{Wolf}, \cite{NW}, \cite{GebSag}, \cite{OS}.

Shareshian and Wachs \cite{SW}\cite{CQSF} introduced a quasisymmetric refinement of Stanley's chromatic symmetric function called the \textit{chromatic quasisymmetric function} of a graph.  Let $G = ([n],E)$ be a graph, and let $\kappa:[n] \rightarrow \PP$ be a proper coloring of $G$.  We say that an edge $\{i,j\}$ of $G$ is an \textit{ascent} of $\kappa$ if $i < j$ and $\kappa(i)<\kappa(j).$ Let $\asc(\kappa)$ denote the number of ascents of $\kappa.$  Then the chromatic quasisymmetric function of $G$ is given by $$X_G({\bf x},t) = \displaystyle \sum_{\kappa \in \kappa(G)} t^{\asc(\kappa)}{\bf x}_\kappa.$$  We view a graph with vertex set $[n]$ as a labeled graph. Note that the chromatic quasisymmetric function of a labeled graph depends on the labeling chosen and not on the isomorphism class of the graph.  We can easily see that setting $t=1$ gives Stanley's chromatic symmetric function.  

In the Shareshian-Wachs chromatic quasisymmetric function of a graph, we can see that the coefficient of $t^j$ for each $j \in \N$ is a quasisymmetric function, so $X_G({\bf x},t) \in QSym_\Q[t],$ where $QSym_\Q[t]$ is the ring of polynomials in $t$ whose coefficients are in the ring of quasisymmetric functions in the variables ${\bf x} = x_1, x_2, \cdots$ with coefficients in $\Q.$  Note that $QSym_\Q[t]$ is equivalent to $QSym_{\Q[t]},$ i.e. the ring of quasisymmetric functions with coefficients in the ring $\Q[t]$ of polynomials in $t$ with coefficients in $\Q.$  In this paper, we tend to view these chromatic quasisymmetric functions as elements of $QSym_\Q[t],$ but we may sometimes view them as elements of $QSym_{\Q[t]}$ when convenient.   

The coefficients of the chromatic quasisymmetric function of a graph do not necessarily have to be symmetric functions.  The chromatic quasisymmetric function of the graph $1 - 2 - 3$ has symmetric coefficients, but the chromatic quasisymmetric function of the graph $2 - 1 - 3$ does not (see \cite[Example 3.2]{CQSF}).  Shareshian and Wachs showed that if $G$ is a natural unit interval graph, i.e. a unit interval graph with a certain natural labeling, then $X_G ({\bf x},t)$ is symmetric, i.e. $X_G({\bf x},t) \in \Lambda_\Q[t].$  They also established that if $G$ is a natural unit interval graph, $X_G({\bf x},t)$ is a palindromic polynomial in $t$; in other words, $X_G({\bf x},t) = \sum_{i = 0}^{|E|} a_i({\bf x})t^i$ so that $a_i({\bf x}) = a_{|E|-i}({\bf x})$ for all $0 \leq i \leq |E|.$  They presented a refinement of the Stanley-Stembridge conjecture for natural unit interval graphs.

\begin{conj} [Shareshian-Wachs \cite{SW} \cite{CQSF}]
Let $G = ([n], E)$ be a natural unit interval graph.  Then the palindromic polynomial $X_G({\bf x},t)$ is $e$-positive and $e$-unimodal.  

In other words, if $X_G({\bf x},t) = \sum_{j = 0}^{|E|} a_j({\bf x})t^j,$ then $a_j({\bf x})$ is $e$-positive for all $j$ and $a_{j+1}({\bf x}) - a_j({\bf x})$ is $e$-positive for all $j \leq \frac{|E|-1}{2}.$  
\end{conj} 

The class of unit interval graphs is equivalent to the class of incomparability graphs of $(3+1)$ and $(2+2)$-free posets, so the class of graphs for the Shareshian-Wachs conjecture is smaller than the class of graphs for the Stanley-Stembridge conjecture.  However, Guay-Pacquet \cite{GP} proved that if the Stanley-Stembridge conjecture holds for $(3+1)$ and $(2+2)$-free posets, then it holds for $(3+1)$-free posets.  Hence, the Shareshian-Wachs conjecture implies the Stanley-Stembridge conjecture. 

There is an important connection between chromatic quasisymmetric functions of natural unit interval graphs and Hessenberg varieties, which was conjectured by Shareshian and Wachs and was proven by Brosnan and Chow \cite{BrosChow} and later by Guay-Paquet \cite{GP2}. Clearman, Hyatt, Shelton, and Skandera \cite{Hecke} found an algebraic interpretation of chromatic quasisymmetric functions of natural unit interval graphs in terms of characters of type A Hecke algebras evaluated at Kazhdan-Lusztig basis elements. Recently, Haglund and Wilson \cite{HW} discovered a connection between chromatic quasisymmetric functions and Macdonald polynomials.

We extend the work of Shareshian and Wachs from labeled graphs to (simple)\footnote{We say a directed graph $\overrightarrow{G} = (V, E)$ is simple if there are no loops, i.e. $(v,v) \notin E$ for all $v \in V,$ and for any distinct vertices $u, v \in V$ there can be at most one edge directed from $u$ to $v.$  Note that we do allow two edges between $u$ and $v$, but the edges must have opposite orientations.} \textit{directed} graphs\footnote{see \hyperref[remark]{Remark \ref*{remark}}}.  For notational convenience, we distinguish an undirected graph, $G$, from a directed graph, $\overrightarrow{G},$ with an arrow.  Throughout this paper, we will refer to the \textit{underlying undirected graph} of a digraph $\overrightarrow{G}$ by which we mean the simple undirected graph obtained by removing the orientation from the edges of $\overrightarrow{G}$ and combining any double edges into single edges.  By a  proper coloring of a digraph, we mean a proper coloring of the underlying undirected graph.

\begin{definition} Let $\overrightarrow{G}=(V,E)$ be a digraph.  Given a proper coloring, $\kappa:V \rightarrow \PP$ of $\overrightarrow{G}$, we define an \textit{ascent} of $\kappa$ to be a directed edge $(v_i,v_j) \in E$ with $\kappa(v_i)<\kappa(v_j),$ and we let $\asc(\kappa)$ denote the number of ascents of $\kappa.$  Then the chromatic quasisymmetric function of $\overrightarrow{G}$ is $$X_{\overrightarrow{G}}({\bf x},t) = \displaystyle \sum_{\kappa \in \kappa(\overrightarrow{G})} t^{\asc(\kappa)}{\bf x}_\kappa.$$
\end{definition}

As with the Shareshian-Wachs chromatic quasisymmetric function, setting $t = 1$ gives Stanley's chromatic symmetric function.  We can easily see that for any digraph, $\overrightarrow{G}$, we have $X_{\overrightarrow{G}}({\bf x},t) \in QSym_\Q[t].$  If we take a labeled graph $G = ([n],E)$ and make a digraph $\overrightarrow{G}$ by orienting each edge from the vertex with smaller label to the vertex with larger label, then $X_G({\bf x},t) = X_{\overrightarrow{G}}({\bf x},t).$  Conversely, every acyclic digraph can be obtained in this way (up to isomorphism).  In other words, the Shareshian-Wachs definition of chromatic quasisymmetric function for labeled graphs is equivalent to our definition of chromatic quasisymmetric function in the case of acyclic digraphs.

As in the case of the Shareshian-Wachs chromatic quasisymmetric functions, our chromatic quasisymmetric functions are not always symmetric.  In this paper, we  introduce a class of digraphs, which we call circular indifference digraphs, and we show that if $\overrightarrow{G}$ is a circular indifference digraph, then $X_{\overrightarrow{G}}({\bf x},t) \in \Lambda_\Q[t].$  It turns out that if we make natural unit interval graphs into digraphs by orienting edges from smaller label to larger label, then these digraphs are included in the class of circular indifference digraphs.  In fact, they are exactly the acyclic circular indifference digraphs. Therefore, our symmetry result implies the symmetry result of Shareshian and Wachs.  The simplest non-acyclic circular indifference digraph is the cycle on $n$ vertices with edges directed cyclically, which we denote $\overrightarrow{C_n}.$  

In \cite{CSF}, Stanley shows that for any graph $G$, $\omega X_G({\bf x})$ is $p$-positive.  In \cite{CQSF}, Shareshian and Wachs conjecture a $p
$-expansion formula for $\omega X_G({\bf x},t)$ when $G$ is a natural unit interval graph, which implies $p$-positivity, and in \cite{Ath}, Athanasiadis proves it.  Here we present a $p$-positivity formula for $\omega X_{\overrightarrow{G}}({\bf x},t)$ when $\overrightarrow{G}$ is \textit{any} digraph such that $X_{\overrightarrow{G}}({\bf x},t)$ is symmetric.  

In this paper, we also introduce an $e$-expansion formula for $X_{\overrightarrow{C_n}}({\bf x},t),$ which is a $t$-analog of a formula of Stanley's for the $e$-expansion of $X_{C_n}({\bf x}),$ where $C_n$ is the undirected cycle.  Using our expansion, we show that $X_{\overrightarrow{C_n}}({\bf x},t)$ is $e$-positive and $e$-unimodal. This gives some evidence for the following generalization of the Shareshian-Wachs $e$-positivity conjecture.

\begin{conj}
\label{e-pos conj}
Let $\overrightarrow{G} = (V,E)$ be a circular indifference digraph.  Then the palindromic\footnote{See \hyperref[palindromic]{Proposition \ref*{palindromic}}}  polynomial $X_{\overrightarrow{G}}({\bf x},t)$ is $e$-positive and $e$-unimodal.  In other words, if $X_{\overrightarrow{G}}({\bf x},t) = \displaystyle \sum_{j = 0}^{|E|} a_j({\bf x})t^j,$ then $a_j({\bf x})$ is $e$-positive for all j and $a_{j+1}({\bf x}) - a_j({\bf x})$ is $e$-positive for all $j \leq \frac{|E|-1}{2}.$  
\end{conj}

This conjecture clearly implies the Shareshian-Wachs conjecture, which in turn implies the Stanley-Stembridge conjecture.  This also addresses a question of Stanley from his paper on chromatic symmetric functions {\cite{CSF}}. He defines a class of graphs, which he calls circular indifference graphs, and speculates that they may be $e$-positive.  These circular indifference graphs are the underlying undirected graphs of our circular indifference digraphs, so setting $t=1$ in \hyperref[e-pos conj]{Conjecture \ref*{e-pos conj}} would give a positive answer to Stanley's question.

This paper is organized as follows.  In \hyperref[Basic Results]{Section \ref*{Basic Results}}, we cover some basic results on chromatic quasisymmetric functions of directed graphs and give a few examples.  In \hyperref[F-basis]{Section \ref*{F-basis}}, we give an expansion of the chromatic quasisymmetric function of all digraphs in terms of Gessel's fundamental quasisymmetric functions.  

In \hyperref[p-basis]{Section \ref*{p-basis}}, we present our $p$-expansion of $\omega X_{\overrightarrow{G}}({\bf x},t)$ for any digraph $\overrightarrow{G}$ such that $X_{\overrightarrow{G}}({\bf x},t)$ is symmetric.  We also give a factorization for the coefficients of $\omega X_{\overrightarrow{C_n}}({\bf x},t)$ in the $p$-basis that involves Eulerian polynomials.  In \hyperref[Symmetry]{Section \ref*{Symmetry}}, we introduce the class of circular indifference digraphs and show that their chromatic quasisymmetric functions are symmetric.  


In \hyperref[e-basis]{Section \ref*{e-basis}}, we discuss some results on the expansion of chromatic quasisymmetric functions of circular indifference digraphs in term of the elementary symmetric basis and provide some support for Conjecture \ref{e-pos conj}. We observe that a generalization of the Shareshian-Wachs refinement of Stanley's result giving a relationship between acyclic orientations of a graph and coefficients in the elementary symmetric function expansion holds in the setting of digraphs.  We also present the formula for the $e$-expansion of $X_{\overrightarrow{C_n}}({\bf x},t).$  We then provide a combinatorial interpretation of the coefficients of the expansion in the $e$-basis of $X_{\overrightarrow{C_n}}({\bf x},t)$ and $X_{\overrightarrow{P_n}}({\bf x},t),$ where $\overrightarrow{P_n}$ is the path on $n$ vertices with all edges oriented the same direction.  

In \hyperref[Graphs]{Appendix \ref*{Graphs}}, we discuss some properties of circular indifference digraphs.  We discuss how these graphs are related to other well-known graphs.  We also explain how this class of graphs relates to the class of natural unit interval graphs studied by Shareshian and Wachs.

\begin{remark} \label{remark}
Most of the results of this paper were presented in the FPSAC extended abstract \cite{FPSAC}.  Here we present their proofs and some additional results.

The idea of extending chromatic quasisymmetric functions to directed graphs was a suggestion made by Richard Stanley to the author after attending a talk on her work on the chromatic quasisymmetric function of the labeled cycle \cite{me2}.

Subsequent to our work, Alexandersson and Panova \cite{AP} independently obtained the symmetry result of \hyperref[Symmetry]{Section \ref*{Symmetry}} and the $e$-expansion results of \hyperref[e-basis]{Section \ref*{e-basis}}. However, their proof of \hyperref[cycle exp]{Theorem~\ref*{cycle exp}}, giving the $e$-expansion of $X_{\overrightarrow{C_n}}({\bf x},t)$, is very different from ours.

A directed graph refinement of Stanley's Tutte symmetric function \cite{StanGar} was recently studied by Awan and Bernardi \cite{AB} and reduces to our definition of the chromatic quasisymmetric function for directed graphs; however, the results of this paper are disjoint from the results of \cite{AB}.

\end{remark}

\section{Basic results}\label{Basic Results}
Let $QSym_\Q$ denote the $\Q$-algebra of quasisymmetric functions in the variables ${\bf x} = x_1, x_2,...$ with coefficients in $\Q$.  See \cite[Chapter 7]{EC2} for a reference.  Let $QSym_{\Q}^n$ denote the $\Q$-submodule of homogeneous quasisymmetric functions in $QSym_{\Q}$ of degree $n$.  Both of the following propositions follow easily from the definition of the chromatic quasisymmetric function of a digraph.

\begin{prop}
For any digraph $\overrightarrow{G} = (V,E)$ with $|V| = n$, we have $X_{\overrightarrow{G}}({\bf x},t) \in QSym_\Q^n[t].$
\end{prop}

\begin{prop}\label{disj}
Let $\overrightarrow{G}$ and $\overrightarrow{H}$ be digraphs on disjoint vertex sets and let $\overrightarrow{G+H}$ denote the graph formed by the disjoint union of $\overrightarrow{G}$ and $\overrightarrow{H}$.  Then $X_{\overrightarrow{G+H}}({\bf x},t) = X_{\overrightarrow{G}}({\bf x},t)X_{\overrightarrow{H}}({\bf x},t).$
\end{prop}

Now let us look at a few examples of digraphs and their corresponding chromatic quasisymmetric functions.

\begin{example}
For any digraph, $\overrightarrow{G},$ on $n$ vertices whose underlying undirected graph is the complete graph, $K_n$, we have that $X_{\overrightarrow{G}}({\bf x},t) = p(t) e_n,$ where $p(t) = \sum_{\kappa} t^{\asc_{\overrightarrow{G}}(\kappa)}$ and $\kappa$ varies over all proper colorings of $\overrightarrow{G}$ using only the colors in $[n].$  From this we can see that $X_{\overrightarrow{G}}({\bf x}, t)$ is $e$-positive.  Specifically if $\overrightarrow{G}$ is acyclic, then $p(t) = [n]_t !$, where $[n]_t = 1+t+ \cdots + t^{n-1}$ and $[n]_t! = [n]_t [n-1]_t \cdots [1]_t$. (See \cite[Example 2.4]{CQSF}.)  If $\overrightarrow{G}$ contains all pairs of double edges, $p(t) = n! t^{(\substack{{n} \\ {2}})}.$ (See \hyperref[Gnn thm]{Theorem \ref*{Gnn thm}}.)
\end{example}

\begin{example}\label{dir path ex}
Let $\overrightarrow{P_n} = (V, E)$ denote the \textit{directed path} on $n$ vertices with vertex set $V = \{v_1, v_2, \cdots, v_n\}$ and edge set $E = \{(v_i, v_{i+1})\mid 1 \leq i <n\}$.  Shareshian and Wachs \cite[Theorem 7.2]{Eul} showed that $$\displaystyle \sum_{n \geq 0} X_{\overrightarrow{P_n}}({\bf x},t)z^n = \frac{\displaystyle \sum_{k \geq 0}e_k({\bf x}) z^k}{1-t \displaystyle \sum_{k \geq 2} [k-1]_t e_k({\bf x}) z^k},$$ which refines a formula of Stanley's for the chromatic symmetric function of the undirected path, $P_n$ \cite[Proposition 5.3]{CSF}.  From this formula, we can see that $X_{\overrightarrow{P_n}}({\bf x},t)$ is symmetric, $e$-positive, and $e$-unimodal \cite[Corollary C.5]{CQSF}.
\end{example}

\begin{example}\label{dir cyc ex}
Let us define the \textit{directed cycle} on $n$ vertices, denoted $\overrightarrow{C_n} = (V, E),$ as the digraph with vertex set $V = \{v_1, v_2, \cdots, v_n\}$ and edge set $E = \{(v_i, v_{i+1}) \mid 1 \leq i <n\} \cup \{(v_n, v_1)\}.$  In \hyperref[cycle exp]{Theorem \ref*{cycle exp}}, we show that \begin{equation} \label{Cn form ref} \displaystyle \sum_{n \geq 2} X_{\overrightarrow{C_n}}({\bf x},t)z^n = \frac{t \displaystyle \sum_{k \geq 2}k[k-1]_t e_k({\bf x}) z^k}{1-t \displaystyle \sum_{k \geq 2} [k-1]_t e_k({\bf x}) z^k},
\end{equation} and hence $X_{\overrightarrow{C_n}}({\bf x},t)$ is symmetric.  In fact, in \hyperref[e-uni]{Corollary \ref*{e-uni}}, we show that the coefficients are $e$-positive and $e$-unimodal. \hyperref[Cn form ref]{Equation (\ref*{Cn form ref})} is a $t$-analog of a formula of Stanley's for the chromatic symmetric function of the undirected cycle, $C_n$ \cite[Proposition 5.4]{CSF}.
\end{example}

Unfortunately, not every orientation of a given graph has a symmetric chromatic quasisymmetric function.  The smallest example of this is the path on 3 vertices.  The orientations of the path are given by 

\begin{center}
\includegraphics[scale=0.5]{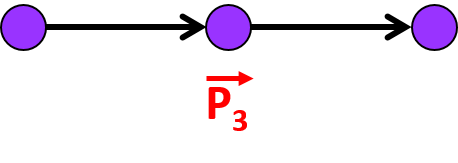}

\includegraphics[scale=0.5]{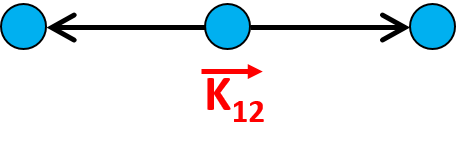} \hspace{0.5 in} \includegraphics[scale=0.5]{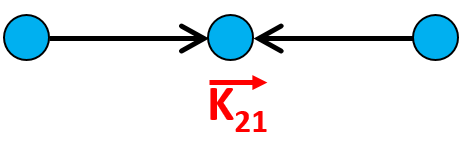}
\end{center}
There is only one orientation ($\overrightarrow{P_3}$) with a symmetric chromatic quasisymmetric function. The other two orientations ($\overrightarrow{K_{12}}$ and $\overrightarrow{K_{21}}$) do not have symmetric chromatic quasisymmetric functions. See \hyperref[AB ex]{Example \ref*{AB ex}}.

On the other hand, there are also graphs that do not admit any orientation whose associated chromatic quasisymmetric function is symmetric.  The graph $K_{31}$ is given by 
\begin{center}
\includegraphics[scale=0.3]{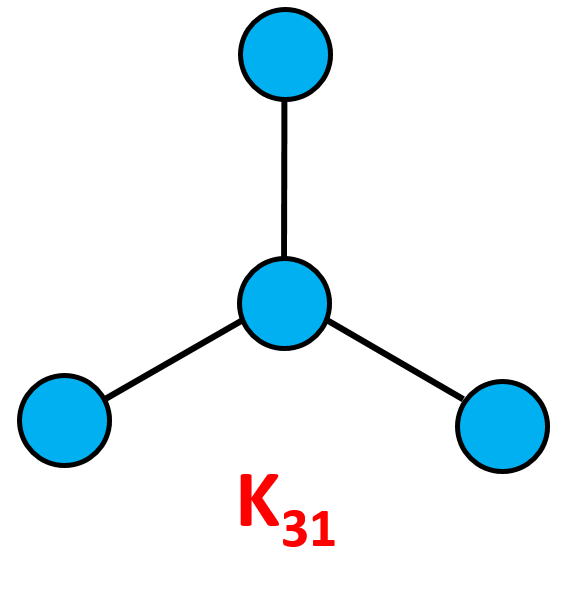}
\end{center}
None of the orientations of $K_{31}$ have chromatic quasisymmetric functions that are symmetric.

Let $\rho : QSym_\Q \rightarrow QSym_\Q$ be the involution defined on the monomial quasisymmetric function basis, $M_{\alpha}$, by $\rho(M_{\alpha}) = M_{\alpha^{rev}}$ for each composition $\alpha.$ Note that every symmetric function is fixed by $\rho.$  We can extend $\rho$ to $QSym_\Z[t]$ by linearity.  Then the next propositions follow easily from \cite[Proposition 2.6, Corollary 2.7, Corollary 2.8]{CQSF}.  Note in \cite{CQSF}, these statements are proven for labeled graphs; however, the same proof works for digraphs.

\begin{prop}
Let $\overrightarrow{G}$ be a digraph on $n$ vertices.  Then $$\rho(X_{\overrightarrow{G}}({\bf x},t)) = \displaystyle \sum_{\kappa \in \kappa(\overrightarrow{G})} t^{\des(\kappa)} {\bf x}_{\kappa},$$ where $\des(\kappa)$ is the number of directed edges $(u,v)$ of $\overrightarrow{G}$ such that $\kappa(u)>\kappa(v)$.

Hence if $X_{\overrightarrow{G}}({\bf x},t)$ is symmetric, then $$X_{\overrightarrow{G}}({\bf x},t) = \displaystyle \sum_{\kappa \in \kappa(\overrightarrow{G})} t^{\des(\kappa)} {\bf x}_{\kappa}.$$
\end{prop}

\begin{prop} \label{palindromic}
If $X_{\overrightarrow{G}}({\bf x},t)$ is symmetric, then $X_{\overrightarrow{G}}({\bf x},t)$ is palindromic in $t$ with center of symmetry $\frac{\mid E(\overrightarrow{G}) \mid}{2}.$
\end{prop}

\section{Expansion in Gessel's fundamental quasisymmetric functions} \label{F-basis}

Shareshian and Wachs give an expansion of $\omega X_G({\bf x},t)$ in terms of Gessel's fundamental quasisymmetric basis \cite[Theorem 3.1]{CQSF} when $G$ is the incomparabilty graph of a poset, which shows that $\omega X_{G}({\bf x},t)$ is $F$-positive.  In this section, we present an $F$-basis expansion of $\omega X_{\overrightarrow{G}}({\bf x},t)$ for \textit{all} digraphs, which shows that $\omega X_{\overrightarrow{G}}({\bf x},t)$ is $F$-positive for \textit{all} digraphs.  In general our formula does not reduce to the formula of Shareshian and Wachs, so this gives another combinatorial description of the coefficients in the $F$-expansion for incomparability graphs of posets.

We may assume without loss of generality that the vertex set of a digraph $\overrightarrow{G}$ is $[n].$  The labeling chosen does not affect the chromatic quasisymmetric function of $\overrightarrow{G},$ as it would for the chromatic quasisymmetric function of a labeled graph defined by Shareshian and Wachs.

Let $\overrightarrow{G} = ([n],E)$ be a digraph and let $\sigma \in \mathfrak{S}_n.$  Define a $\overrightarrow{G}$\textit{-inversion} of $\sigma$ as a directed edge $(u,v)$ of $\overrightarrow{G}$ such that $\sigma^{-1}(u) > \sigma^{-1}(v),$ i.e. $v$ precedes $u$ in $\sigma.$ Notice that a $\overrightarrow{G}$-inversion does not need to be a usual inversion; if $v < u,$ then it will not be. Let $\inv_{\overrightarrow{G}}(\sigma)$ be the number of $\overrightarrow{G}$-inversions of $\sigma.$

Now let $G = ([n],E)$ be an undirected graph and let $\sigma = \sigma_1 \sigma_2 \cdots \sigma_n \in \mathfrak{S}_n$.  For each $x \in [n],$ define the $(G, \sigma)$\textit{-rank} of $x$, denoted $\rank_{(G,\sigma)}(x),$ as the length of the longest subword $\sigma_{i_1}\sigma_{i_2}\cdots\sigma_{i_k}$ of $\sigma$ such that $\sigma_{i_k} = x$ and for each $1 \leq j < k$, $\{\sigma_{i_j},\sigma_{i_{j+1}}\} \in E$.  We say $\sigma$ has a \textit{$G$-descent} at $i$ with $1 \leq i < n$ if either of the following conditions holds:

\begin{itemize}
 \item $\rank_{(G, \sigma)}(\sigma_i) > \rank_{(G, \sigma)}(\sigma_{i+1})$
 \item $\rank_{(G, \sigma)}(\sigma_i) = \rank_{(G, \sigma)}(\sigma_{i+1})$ and $\sigma_i > \sigma_{i+1}.$  
\end{itemize}
Let $\DES_{G}(\sigma)$ be the set of $G$-descents of $\sigma.$

For example, let $G = C_8$ be the cycle on 8 vertices labeled with $[8]$ in cyclic order.  Let $\sigma = 25413786 \in \mathfrak{S}_8$.  By attaching the $(G, \sigma)$-rank of each letter as a superscript, we get $2^1 5^1 4^2 1^2 3^3 7^1 8^3 6^2.$  We can see from this that $\DES_G(\sigma) = \{3, 5, 7\}.$  

\begin{thm} \label{F-basis thm}
Let $\overrightarrow{G} = ([n],E)$ be a digraph.  Then \begin{equation}\label{F-basis eqn} \omega X_{\overrightarrow{G}}({\bf x},t) = \displaystyle \sum_{\sigma \in \mathfrak{S}_n} F_{n, \DES_{G}(\sigma)}({\bf x})t^{\inv_{\overrightarrow{G}}(\sigma)},
\end{equation} where $F_{n,S}({\bf x})$ is Gessel's fundamental quasisymmetric function and $G$ is the underlying undirected graph of $\overrightarrow{G}.$ 
\end{thm}

\begin{proof}
The first half of this proof closely follows the proof of \cite[Theorem 3.1]{CQSF}.  The second part of the proof is very different.

Let $G$ be the underlying undirected graph of $\overrightarrow{G}$ and let $G_{\bar{a}}$ be an acyclic orientation of the graph $G$, possibly different from the given orientation on $\overrightarrow{G}.$  Then define $\asc(G_{\bar{a}})$ to be the number of edges of $G_{\bar{a}}$ whose orientation matches the orientation of $\overrightarrow{G}.$  For each acyclic orientation $G_{\bar{a}}$, we will let $C(G_{\bar{a}})$ be the set of proper colorings, $\kappa,$ of $G$ such that if $(i,j)$ is a directed edge of $G_{\bar{a}}$, then $\kappa(i) < \kappa(j).$  It is clear that \begin{equation}\label{Pptn eqn}X_{\overrightarrow{G}}({\bf x},t) = \displaystyle \sum_{G_{\bar{a}} \in AO(G)} t^{\asc(G_{\bar{a}})} \displaystyle \sum_{\kappa \in C(G_{\bar{a}})} {\bf x}_{\kappa}, 
\end{equation} where $AO(G)$ is the set of acyclic orientations of $G$.

From each acyclic orientation, $G_{\bar{a}}$, we can create a poset, $P_{\bar{a}},$ on $[n]$ by letting $i <_{P_{\bar{a}}} j$ if there is an edge from $i$ to $j$ in $G_{\bar{a}}$ and extending transitively.  We define a labeling of $P_{\bar{a}}$ as a bijection from $P_{\bar{a}}$ to $[n].$  Now we give $P_{\bar{a}}$ a decreasing labeling $w_{\bar{a}}:P_{\bar{a}} \rightarrow [n],$ i.e. if $x <_{P_{\bar{a}}} y,$  then $w_{\bar{a}}(x) > w_{\bar{a}}(y).$  Let $L(P_{\bar{a}},w_{\bar{a}})$ be the set of linear extensions of $P_{\bar{a}}$ with the labeling $w_{\bar{a}}.$  For any subset $S \subseteq [n-1],$ define $n-S = \{i \mid n-i \in S\}.$ Then by the theory of P-partitions, we have 
\begin{equation} \label{ptn eqn}
\displaystyle \sum_{\kappa \in C(G_{\bar{a}})} {\bf x}_{\kappa} = \displaystyle \sum_{\sigma \in L(P_{\bar{a}}, w_{\bar{a}})} F_{n,n-\DES(\sigma)},
\end{equation} where $\DES(\sigma)$ is the usual descent set of a permutation, i.e. $\DES(\sigma) = \{i \in [n-1] \mid \sigma(i) > \sigma(i+1)\}.$

Let $e:P_{\bar{a}} \rightarrow [n]$ be the identity map, i.e. the map that takes each element of $P_{\bar{a}}$ to its original label. Then $L(P_{\bar{a}}, e)$ is the set of linear extensions of $P_{\bar{a}}$ with its original labeling, $e.$  For $\sigma \in L(P_{\bar{a}}, e),$ let $w_{\bar{a}} \sigma$ denote the product of $w_{\bar{a}}$ and $\sigma$ in $\mathfrak{S}_n.$  For $\sigma \in \mathfrak{S}_n,$ we have $w_{\bar{a}} \sigma \in L(P_{\bar{a}}, w_{\bar{a}})$ if and only if $\sigma \in L(P_{\bar{a}}, e).$ So \hyperref[ptn eqn]{(\ref*{ptn eqn})} becomes 
$$\displaystyle \sum_{\kappa \in C(G_{\bar{a}})} {\bf x}_{\kappa} = \displaystyle \sum_{\sigma \in L(P_{\bar{a}}, e)} F_{n,n-\DES(w_{\bar{a}} \sigma)}.$$
Combining this with \hyperref[Pptn eqn]{(\ref*{Pptn eqn})} gives us that $$X_{\overrightarrow{G}}({\bf x},t) = \displaystyle \sum_{G_{\bar{a}} \in AO(G)} t^{\asc(G_{\bar{a}})} \displaystyle \sum_{\sigma \in L(P_{\bar{a}}, e)} F_{n,n-\DES(w_{\bar{a}}\sigma)}.$$
Since each $\sigma \in \S_n$ is a linear extension of a unique acyclic orientation, $G_{\bar{a}(\sigma)},$ of G, we can rewrite this as $$X_{\overrightarrow{G}}({\bf x},t) = \displaystyle \sum_{\sigma \in \S_n} t^{\asc(G_{\bar{a}(\sigma)})} F_{n,n-\DES(w_{\bar{a} (\sigma)}\sigma)},$$ where $w_{\bar{a}(\sigma)}$ is a decreasing labeling of $P_{\bar{a}(\sigma)}.$

For $\sigma \in \mathfrak{S}_n,$ let $\ASC(\sigma)$ denote the usual ascent set of a permutation, i.e. $\ASC(\sigma) = \{ i \in [n-1] \mid \sigma(i) < \sigma(i+1)\}.$  Also define $\sigma^{\rm rev} \in \mathfrak{S}_n$ by letting $\sigma^{\rm rev}(i) = \sigma(n+1-i)$ for all $i$. It is not hard to see that $\asc(G_{\bar{a}(\sigma)}) = \inv_{\overrightarrow{G}}(\sigma^{\rm rev})$ and $n-\DES(w_{\bar{a}(\sigma)}\sigma) = \ASC((w_{\bar{a}(\sigma)}\sigma)^{\rm rev}),$ so we have $$X_{\overrightarrow{G}}({\bf x},t) = \displaystyle \sum_{\sigma \in \S_n} t^{\inv_{\overrightarrow{G}}(\sigma^{\rm rev})}F_{n,\ASC((w_{\bar{a}(\sigma)}\sigma)^{\rm rev})}.$$

Then by reversing $\sigma$ and letting $\tilde{w}_{\bar{a}(\sigma)}$ be an increasing labeling of $P_{\bar{a}(\sigma)},$ we have that $$X_{\overrightarrow{G}}({\bf x},t) = \displaystyle \sum_{\sigma \in \S_n}t^{\inv_{\overrightarrow{G}}(\sigma)}F_{n,\ASC(\tilde{w}_{\bar{a}(\sigma)}\sigma)}.$$

Finally applying the involution $\omega$ to both sides of the equation gives us $$\omega X_{\overrightarrow{G}}({\bf x},t) = \displaystyle \sum_{\sigma \in \S_n}t^{\inv_{\overrightarrow{G}}(\sigma)}F_{n,\DES(\tilde{w}_{\bar{a}(\sigma)}\sigma)}.$$

Up to this point, $\tilde{w}_{\bar{a}(\sigma)}$ has been any increasing labeling of $P_{\bar{a}(\sigma)},$ but now we will make it a specific one. Note that we will refer to the original labeling of the vertices of $G$ as the $G$-labeling of the graph.  For each acyclic orientation $G_{\bar{a}}$ and for each vertex $v,$ define $\rank_{\bar{a}}(v)$ as the length of the longest chain of $P_{\bar{a}}$ from a minimal element of $P_{\bar{a}}$ to $v.$ We say $v$ is a $\rank_{\bar{a}}\ i$ element if $\rank_{\bar{a}}(v) = i.$ To determine the labeling $\tilde{w}_{\bar{a}(\sigma)},$ first we label the $\rank_{\bar{a}(\sigma)}\ 0$ element with the smallest $G$-label as 1.  Then we label the $\rank_{\bar{a}(\sigma)}\ 0$ element with the next smallest $G$-label as 2.  We continue this process until all $\rank_{\bar{a}(\sigma)}\ 0$ elements are labeled. Then we repeat this process with the $\rank_{\bar{a}(\sigma)}\ 1$ elements and continue inductively until all elements are labeled.

Notice that for all $x \in [n],$ we have $\rank_{(G, \sigma)}(x) = \rank_{\bar{a}(\sigma)}(x) + 1.$ So using the labeling $\tilde{w}_{\bar{a}(\sigma)}$ constructed above, if $i$ is a descent of $\tilde{w}_{\bar{a}(\sigma)}\sigma,$ then $\sigma(i+1)$ was labeled before $\sigma(i)$ in the labeling $\tilde{w}_{\bar{a}(\sigma)}.$  Then either $\sigma(i+1)$ has a smaller $\bar{a}(\sigma)$-rank than $\sigma(i)$ or they have the same $\bar{a}(\sigma)$-rank and $\sigma(i)>\sigma(i+1).$  But in either case, $i$ is also a $G$-descent of $\sigma.$  A similar argument shows that if $i$ is a $G$-descent of $\sigma,$ then $i$ is also a descent of $\tilde{w}_{\bar{a}(\sigma)}\sigma.$  Hence $\DES(\tilde{w}_{\bar{a}(\sigma)}\sigma) = \DES_{G}(\sigma),$ and the theorem is proven.
\end{proof}

Note that our formula requires that $\overrightarrow{G}$ be labeled with $[n]$; however, the expansion of $X_{\overrightarrow{G}}({\bf x},t)$ in the $F$-basis is the same for any choice of labeling.

\begin{example} \label{AB ex}
Let us give labelings to the digraphs $\overrightarrow{P_3}$, $\overrightarrow{K_{12}}$, and $\overrightarrow{K_{21}}$ mentioned in \hyperref[Basic Results]{Section \ref*{Basic Results}} as follows:  

\begin{center}
\includegraphics[scale=0.5]{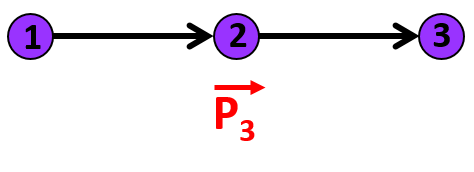}

\includegraphics[scale=0.5]{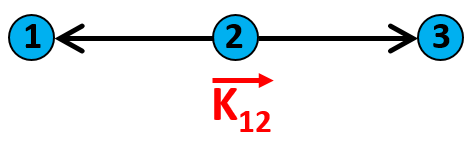} \hspace{0.5 in} \includegraphics[scale=0.5]{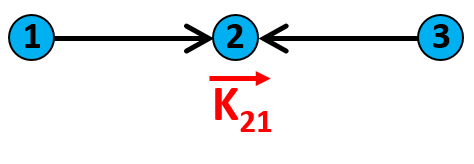}
\end{center}

To expand the chromatic quasisymmetric function of each of these in the $F$-basis, we need to calculate $\DES_G(\sigma)$ and $\inv_{\overrightarrow{G}}(\sigma)$ for each $\sigma \in \mathfrak{S}_3.$  Since $\DES_G(\sigma)$ only depends on the underlying undirected graph, this will be the same for each of $\overrightarrow{P_3}$, $\overrightarrow{K_{12}}$, and $\overrightarrow{K_{21}}$.  The calculations are shown in the chart below.

\begin{center}
\begin{tabular}{|c|c|c|c|c|}
\hline
$\sigma$ & $\DES_G(\sigma)$ & $\inv_{\overrightarrow{P_3}}(\sigma)$ &  $\inv_{\overrightarrow{K_{12}}}(\sigma)$ & $\inv_{\overrightarrow{K_{21}}}(\sigma)$\\
\hline \hline  
123 & $\emptyset$ & 0 & 1 & 1 \\
\hline
132 & $\emptyset$ & 1 & 2 & 0 \\
\hline
213 & $\emptyset$ & 1 & 0 & 2 \\
\hline
231 & $\{2\}$ & 1 & 0 & 2 \\
\hline
312 & $\{1\}$ & 1 & 2 & 0 \\
\hline
321 & $\emptyset$ & 2 & 1 & 1 \\
\hline
\end{tabular}
\end{center}

Using this, we have that $$\omega X_{\overrightarrow{P_3}}({\bf x}, t) = (F_{3, \emptyset}) + t(2F_{3, \emptyset} + F_{3, \{1\}} +F_{3, \{2\}}) +t^2(F_{3, \emptyset}),$$ $$\omega X_{\overrightarrow{K_{12}}}({\bf x},t) = (F_{3, \emptyset} + F_{3, \{2\}}) + t(2F_{3, \emptyset}) + t^2(F_{3, \emptyset} + F_{3, \{1\}}),$$ $$\omega X_{\overrightarrow{K_{21}}}({\bf x}, t) = (F_{3, \emptyset} + F_{3, \{1\}}) + t(2F_{3, \emptyset}) + t^2(F_{3, \emptyset} + F_{3, \{2\}}).$$

From this we see that $X_{\overrightarrow{P_3}}({\bf x}, t)$ is symmetric and palindromic, but $X_{\overrightarrow{K_{12}}}({\bf x}, t)$ and $X_{\overrightarrow{K_{21}}}({\bf x}, t)$ are neither.

\end{example}

By specializing \hyperref[F-basis eqn]{(\ref*{F-basis eqn})}, we obtain a $t$-analog of a result of Chung and Graham \cite[Theorem 2]{Chung} on the chromatic polynomial of a graph, which we became aware of after obtaining our results.  Let us define the $t$\textit{-analog of the chromatic polynomial of a digraph} $\overrightarrow{G}$ as $$\chi_{\overrightarrow{G}}(m, t) = \sum_{\kappa \in \kappa_m(\overrightarrow{G})} t^{\asc(\kappa)},$$ where $\kappa_m(\overrightarrow{G})$ is the set of proper colorings of $\overrightarrow{G}$ using only colors in $[m].$  From the definition, we can see that for any $m \in \PP,$ $$\chi_{\overrightarrow{G}}(m, t) = X_{\overrightarrow{G}}(1^m, t).$$  Also, if we set $t = 1,$ we see that $$\chi_{\overrightarrow{G}}(m, 1) = \chi_{G}(m),$$ where $G$ is the underlying undirected graph of $\overrightarrow{G}$ and $\chi_{G}(m)$ is the chromatic polynomial.

For $k, l \in \PP$ and a digraph $\overrightarrow{G} = ([n],E)$ with underlying undirected graph, $G,$ let $\delta_{\overrightarrow{G}}(k,l)$ denote the number of permutations $\sigma \in \mathfrak{S}_n$ such that $|\DES_G(\sigma)| = k$ and $\inv_{\overrightarrow{G}}(\sigma) = l$.  

\begin{cor}[t=1 case {\cite[Theorem 2]{Chung}}]
Let $\overrightarrow{G} = ([n],E)$ be a digraph on $n$ vertices.  Then $$\chi_{\overrightarrow{G}}(m, t) = \sum_{k, l \geq 0} \delta_{\overrightarrow{G}}(k,l)(\substack{{m+k}\\ {n}})t^l.$$ Consequently\footnote{Athanasiadis \cite{Ath2} observed this when $G = Inc(P)$}, this is a polynomial in $m$ whose coefficients are palindromic polynomials in $t.$
\end{cor}

\begin{proof}
We know that for any $S \subseteq [n-1]$ with $|S| = k,$ we have $F_{n,S}(1^m) = (\substack{{m+n-1-k} \\ {n}})$ (See \cite[Section 7.19]{EC2}).  Applying $\omega$ to both sides of \hyperref[F-basis eqn]{(\ref*{F-basis eqn})} gives us that $X_{\overrightarrow{G}}({\bf x},t) = \displaystyle \sum_{\sigma \in \mathfrak{S}_n} F_{n, [n-1]\backslash \DES_{G}(\sigma)}({\bf x})t^{\inv_{\overrightarrow{G}}(\sigma)}.$  Then we have
\begin{align*}
\chi_{\overrightarrow{G}}(m,t) &= X_{\overrightarrow{G}}(1^m, t)\\
&= \displaystyle \sum_{\sigma \in \mathfrak{S}_n} F_{n, [n-1]\backslash \DES_{G}(\sigma)}(1^m)t^{\inv_{\overrightarrow{G}}(\sigma)}\\
&= \displaystyle \sum_{\sigma \in \mathfrak{S}_n} (\substack{{m +|\DES_G(\sigma)|}\\ {n}}) t^{\inv_{\overrightarrow{G}}(\sigma)}\\
&= \sum_{k, l \geq 0} \delta_{\overrightarrow{G}}(k,l)(\substack{{m+k}\\ {n}})t^l.
\end{align*}
\end{proof}

\section{Expansion in the power sum symmetric function basis} \label{p-basis}

In \cite{CSF}, Stanley shows that for \textit{any} graph $G$, the symmetric function $\omega X_{G}({\bf x})$ is $p$-positive.  Since not every graph has a symmetric chromatic quasisymmetric function, here we restrict ourselves to graphs that do.  In this section, we establish $p$-positivity for all symmetric $\omega X_{\overrightarrow{G}}({\bf x},t)$ by deriving a $p$-expansion formula.  In \hyperref[Symmetry]{Section \ref*{Symmetry}}, we introduce a class of digraphs with symmetric chromatic quasisymmetric functions, which includes natural unit interval graphs as well as the directed cycle, thereby extending the symmetry result of Shareshian and Wachs.  Our $p$-expansion formula does not reduce to the Shareshian-Wachs-Athanasiadis formula \cite{CQSF} \cite{Ath} for natural unit interval graphs mentioned in the introduction. It reduces to a new formula.

Let $G = ([n],E)$ be an undirected graph and let $w = w_1 w_2 \cdots w_k$ be a word with distinct letters in $[n].$  We say $w_j$ with $1< j \leq k$ is a \textit{$G$-isolated letter} of $w$ if there is no $w_i$ with $1 \leq i <j$ such that $\{w_i, w_j\} \in E.$  


Now for any undirected graph $G=([n],E)$ and any partition $\lambda \vdash n$ with $\lambda = (\lambda_1, \lambda_2, \cdots, \lambda_l)$, we define $N_{\lambda}(G)$ as the set of all $\sigma \in \mathfrak{S}_n$ such that when $\sigma$ is divided up into contiguous segments $\alpha_1, \alpha_2, \cdots, \alpha_l$ of sizes $\lambda_1, \lambda_2, \cdots, \lambda_l,$ each $\alpha_i$ has no $G$-isolated letters and contains no $G$-descents of $\sigma$. Note that the $G$-descents here are determined by the entire $\sigma$ and cannot be determined by looking at the $\alpha_i$'s individually.

Let $G = C_8,$ the cycle on 8 vertices labeled cyclically with $[8]$ and let $ \sigma = 43587162 \in \mathfrak{S}_8.$ Then attaching the $(G, \sigma)$-rank to each letter gives $4^1 3^2 5^2 8^1 7^2 1^2 6^3 2^3$ and hence $\DES_G(\sigma) = \{3, 5, 7\}.$  If $\lambda = 3, 2, 2, 1,$ then $\alpha_1 = 435, \alpha_2 = 87, \alpha_3 = 16, \alpha_4 = 2.$  None of the $\alpha_i$ contain any $G$-descents; however, in $\alpha_3,$ 6 is a $G$-isolated letter, so $\sigma \notin N_{\lambda}(G).$  However if $\lambda = 3, 2, 1, 1, 1,$ then $\sigma \in N_{\lambda}(G).$

For each $\lambda \vdash n,$ let $z_\lambda = \prod_i m_i(\lambda)! i^{m_i(\lambda)},$ where $m_i(\lambda)$ is the multiplicity of $i$ in $\lambda$ for each $i$.

\begin{thm} \label{p thm}
Let $\overrightarrow{G}$ be a digraph such that $X_{\overrightarrow{G}}({\bf x},t)$ is symmetric. Then $\omega X_{\overrightarrow{G}}({\bf x},t)$ is p-positive.  In fact,
\begin{equation}\label{p-exp}\omega X_{\overrightarrow{G}}({\bf x},t) = \displaystyle \sum_{\lambda \vdash n} z_{\lambda}^{-1} p_{\lambda}({\bf x}) \displaystyle \sum_{\sigma \in N_{\lambda}(G)}t^{\inv_{\overrightarrow{G}}(\sigma)},
\end{equation} where $G$ is the underlying undirected graph of $\overrightarrow{G}.$
\end{thm}

Let $\lambda \vdash n$ with $\lambda = (\lambda_1, \lambda_2,...,\lambda_l).$  Define $s_i = \lambda_1 + \lambda_2 +...+\lambda_i,$ for $1 \leq i \leq l$ and $s_0 = 0.$  A set $A \subseteq [n-1]$ is $\lambda$\textit{-unimodal} if for $0 \leq i < l,$ the intersection of $A$ with each set of the form $\{s_i +1,...,s_{i+1}-1\}$ is a prefix of the latter. Additionally, define $S(\lambda) = \{s_1, s_2, \cdots, s_{l-1}\}.$  We will need the following result for our proof, which is implicit in the work of Adin and Roichman \cite[Theorem 3.6]{Adin}, and stated explicitly and proved by Athanasiadis. 

\begin{prop}[Athansiadis {\cite[Proposition 3.2]{Ath}}]\label{Ath prop}
Let $X({\bf x}) \in \Lambda_R^n$ be a homogeneous symmetric function of degree $n$ over a commutative $\Q$-algebra $R$ and suppose that $$X({\bf x}) = \sum_{S \subseteq [n-1]} a_S F_{n,S}({\bf x})$$ for some $a_S \in R.$  Then $$X({\bf x}) = \sum_{\lambda \vdash n}z_{\lambda}^{-1}p_{\lambda}({\bf x}) \sum_{S \in U_\lambda}(-1)^{\mid S\backslash S(\lambda)\mid}a_S,$$ where $U_\lambda$ is the set of $\lambda$-unimodal subsets of $[n-1].$
\end{prop}

In Athanasiadis's proof \cite[Theorem 3.1]{Ath} of the $p$-expansion formula conjectured by Shareshian and Wachs, he uses the $F$-basis decomposition for natural unit inteveral graphs given by Shareshian and Wachs \cite[Theorem 3.1]{CQSF} involving $P$-descents, \hyperref[Ath prop]{Proposition \ref*{Ath prop}}, and a formula for the coefficient of $\frac{1}{n}p_n$ \cite[Lemma 7.4]{CQSF}.  In the following proof, we use our $F$-basis decomposition \hyperref[F-basis thm]{Theorem \ref*{F-basis thm}} for all digraphs, \hyperref[Ath prop]{Proposition \ref*{Ath prop}}, and a sign-reversing involution.

\begin{proof}[Proof of {\hyperref[p thm]{Theorem \ref*{p thm}}}]
Combining \hyperref[Ath prop]{Proposition \ref*{Ath prop}} with our $F$-basis expansion \hyperref[F-basis thm]{Theorem \ref*{F-basis thm}}, we have that \begin{equation} \label{inv eqn} \omega X_{\overrightarrow{G}}({\bf x},t) = \displaystyle \sum_{\lambda \vdash n} z_{\lambda}^{-1} p_{\lambda} \displaystyle \sum_{\substack{{\sigma \in \S_n}\\{\DES_{G}(\sigma) \in U_{\lambda}}}} (-1)^{|\DES_{G}(\sigma)\backslash S(\lambda)|} t^{\inv_{\overrightarrow{G}}(\sigma)},
\end{equation} where $U_{\lambda}$ is the set of $\lambda$-unimodal sets on $[n-1]$.

Let us define the set $$D_{\lambda}(G) := \{\sigma \in \S_n \mid \DES_{G}(\sigma) \in U_{\lambda}\}$$ for each $\lambda \vdash n.$  Note that $N_{\lambda}(G) \subseteq D_{\lambda}(G).$ In order to prove the theorem, we will construct for each $\lambda \vdash n$ a sign-reversing, $\inv_G(\sigma)$-preserving involution, $\gamma_\lambda$, on $D_{\lambda}(G)\backslash N_{\lambda}(G).$  That is $\gamma_{\lambda}:D_{\lambda}(G)\backslash N_{\lambda}(G) \rightarrow D_{\lambda}(G)\backslash N_{\lambda}(G)$ will satisfy the following for all $\sigma \in D_{\lambda}(G) \backslash N_{\lambda}(G)$:
\begin{itemize}
\item $\gamma_{\lambda}^2(\sigma) = \sigma$,
\item $\gamma_{\lambda}$ changes $\mid \DES_{G}(\sigma) \backslash S(\lambda)\mid$ by 1,
\item $\inv_{\overrightarrow{G}}(\sigma) = \inv_{\overrightarrow{G}}(\gamma_{\lambda}(\sigma))$.
\end{itemize}

Now let us fix some useful notation.  Let $\sigma = \sigma_1 \sigma_2 \cdots \sigma_n \in \mathfrak{S}_n.$  We define a total order, $<_G$, on $[n]$ by $x <_G y$ if $\rank_{(G, \sigma)}(x) < \rank_{(G, \sigma)}(y)$ or if $\rank_{(G, \sigma)}(x) = \rank_{(G, \sigma)}(y)$ and $x < y.$  Using this notation, there is a $G$-descent of $\sigma$ at $i$ if and only if $\sigma_i >_G \sigma_{i+1}.$

Fix $\lambda$ and let $\sigma \in D_{\lambda}(G) \backslash N_{\lambda}(G).$  Break $\sigma$ up into contiguous segments of sizes $\lambda_1, \lambda_2, \cdots \lambda_l$ called $\alpha_1, \alpha_2, \cdots \alpha_l.$  Then let $\alpha_i$ be the first segment of $\sigma$ that either has a $G$-isolated letter or a $G$-descent.  Since $\DES_{G}(\sigma) \in U_{\lambda}$, there must exist a unique $k$ such that $s_{i-1} + 1 \leq k \leq s_i$ and $\alpha_i$ is of the form $\alpha_i = {\sigma}_{s_{i-1}+1} \ {\sigma}_{s_{i-1}+2} \cdots {\sigma}_{k-1} \ {\sigma}_k \ {\sigma}_{k+1} \cdots {\sigma}_{s_i},$  where ${\sigma}_{s_{i-1}+1} >_G {\sigma}_{s_{i-1}+2} >_G \cdots >_G {\sigma}_{k-1} >_G {\sigma}_k <_G {\sigma}_{k+1} <_G \cdots <_G {\sigma}_{s_i}$.

Define the involution by setting $\gamma_{\lambda}(\sigma) := \alpha_1 \alpha_2 \cdots \tilde{\alpha_i} \cdots \alpha_l$, where $\tilde{\alpha_i}$ is obtained from $\alpha_i$ by considering the following cases using the $k$ from the previous paragraph:

\ 

\textbf{Case 1a:} $ s_{i-1} +1 < k <s_i,$ and $\sigma_{k-1}<_G \sigma_{k+1}.$

\textbf{Case 1b:} $k = s_i.$

In both cases, let $\tilde{\alpha_i}$ be obtained from $\alpha_i$ by switching $\sigma_{k-1}$ and $\sigma_k.$  Since $\sigma_{k-1} >_G \sigma_{k},$ there cannot be an edge between them.  If there were, since $\sigma_{k-1}$ precedes $\sigma_k$ in $\sigma$, then $\rank_{(G,\sigma)}(\sigma_{k-1}) < \rank_{(G,\sigma)}(\sigma_k),$ but then $\sigma_{k-1} <_G \sigma_k,$ which is a contradiction.  Therefore switching $\sigma_k$ and $\sigma_{k-1}$ doesn't change the number of $\overrightarrow{G}$-inversions, but it decreases the number of $G$-descents by 1. Notice that $\gamma_{\lambda}(\sigma) \in D_{\lambda}(G) \backslash N_{\lambda}(G).$

\

\textbf{Case 2a:} $k = s_{i-1}+1$ and $\{{\sigma}_k, {\sigma}_{k+1}\} \notin E(G).$

\textbf{Case 2b:} $ s_{i-1} +1 < k <s_i,$ ${\sigma}_{k-1} >_G {\sigma}_{k+1}$ and $\{{\sigma}_k, {\sigma}_{k+1}\} \notin E(G).$

For both of these cases, let $\tilde{\alpha_i}$ be obtained from $\alpha_i$ by switching ${\sigma}_k$ and ${\sigma}_{k+1}.$  This will create one new $G$-descent between $\sigma_k$ and $\sigma_{k+1}$, but will not affect the number of $\overrightarrow{G}$-inversions since there is no edge between ${\sigma}_k$ and ${\sigma}_{k+1}$.  Notice that $\gamma_{\lambda}(\sigma) \in D_{\lambda}(G) \backslash N_{\lambda}(G).$

\ 

Now define $\sigma_m$ as the largest (or rightmost) $G$-isolated letter in $\alpha_i$ such that $m>k.$   If there are no $G$-isolated letters after $\sigma_k,$ then define $\sigma_m = 0.$ 

\ 

\textbf{Case 3a:} $s_{i-1} +1 < k <s_i,$ $\sigma_{k-1} >_G \sigma_{k+1}$, $\{\sigma_k, \sigma_{k+1}\} \in E(G)$, $\sigma_m \neq 0,$ and $\sigma_{s_{i-1}+1} <_G \sigma_m.$  

\textbf{Case 3b:} $k = s_{i-1} +1$ and $\{\sigma_k, \sigma_{k+1}\} \in E(G).$

In case 3b since $\alpha_i$ has no $G$-descents, it must have a $G$-isolated letter, so $a_m \neq 0.$ In both cases, obtain $\tilde{\alpha_i}$ by moving $\sigma_m$ before $\sigma_{s_{i-1}+1}.$  Since $\sigma_m$ is a $G$-isolated letter and thus is not connected to any letter before it in $\alpha_i$, this rearrangement will not affect the number of $\overrightarrow{G}$-inversions, but it will create one new $G$-descent between  $\sigma_m$ and $\sigma_{s_{i-1}+1}$. Notice that $\gamma_{\lambda}(\sigma) \in D_{\lambda}(G) \backslash N_{\lambda}(G).$

\ 

\textbf{Case 4a:} $ s_{i-1} +1 < k <s_i,$ $\sigma_{k-1} >_G \sigma_{k+1},$ $\{\sigma_k, \sigma_{k+1}\} \in E(G),$ and $\sigma_m=0.$

\textbf{Case 4b:} $ s_{i-1} +1 < k <s_i,$ $\sigma_{k-1} >_G \sigma_{k+1},$ $\{\sigma_k, \sigma_{k+1}\} \in E(G),$ and $\sigma_m <_G \sigma_{s_{i-1}+1}.$  

Then move $\sigma_{s_{i-1}+1}$ to the first spot after $\sigma_k$ that will not create a new $G$-descent.  We see that this reduces the number of $G$-descents by 1.  Now notice that wherever we finally place $\sigma_{s_{i-1}+1},$ all the $\sigma_j$ that come before this position must satisfy $\sigma_j <_G \sigma_{s_{i-1}+1}.$  It follows that there is no edge between $\sigma_{s_{i-1}+1}$ and $\sigma_j$ since if there were then we would have $\rank_{(G,\sigma)}(\sigma_j) > \rank_{(G, \sigma)}(\sigma_{s_{i-1}+1})$.  Hence this rearrangement does not affect the number of $\overrightarrow{G}$-inversions. Notice that $\gamma_{\lambda}(\sigma) \in D_{\lambda}(G) \backslash N_{\lambda}(G).$

\ 

Now notice that we have covered all cases and these cases are mutually exclusive.  We leave it to the reader to check that Case 1 and Case 2 will reverse each other, as will Case 3 and Case 4, so this is the involution we were looking for.  

Then the only elements of $D_{\lambda}(G)$ that remain in the formula \hyperref[inv eqn]{\ref*{inv eqn}} are those of $N_{\lambda}(G).$  Since these permutations have all their $G$-descents in $S(\lambda)$ by definition, the theorem is proven.
\end{proof}

In \cite[Proposition 7.8]{CQSF}, Shareshian and Wachs showed that when $G$ is a natural unit interval graph, the coefficient of each $z_\lambda^{-1}p_{\lambda}$ in $\omega X_{G}({\bf x},t)$ factors. Though the coefficients do not generally factor in the digraph case, we show in \hyperref[Cn p thm]{Theorem \ref*{Cn p thm}} below that the coefficient of each $z_\lambda^{-1}p_{\lambda}$ in $\omega X_{\overrightarrow{G}}({\bf x},t)$ does have a nice factorization involving the Eulerian polynomials when $\overrightarrow{G}$ is the directed cycle, $\overrightarrow{C_n},$ as defined in \hyperref[dir cyc ex]{Example \ref*{dir cyc ex}}.  We show in \hyperref[Symmetry]{Section \ref*{Symmetry}} that $X_{\overrightarrow{C_n}}({\bf x},t)$ is symmetric.

Let $\sigma = \sigma_1 \sigma_2 \cdots \sigma_n \in \mathfrak{S}_n$ and define $\asc(\sigma) = |\{i \mid \sigma_i < \sigma_{i+1}\}|.$  Define $\exc(\sigma) = |\{i \mid \sigma_i > i\}|.$  Let us define the \textit{Eulerian polynomial}, $A_n(t),$ by the formula $A_n(t) =\sum_{\sigma \in \mathfrak{S}_n} t^{\asc(\sigma)}.$  The following lemma is a special case of \cite[Theorem 3.1]{Brent} but is proven here for completeness.

\begin{lemma} \label{eul lem}
Let $A_k(t)$ denote the Eulerian polynomial. Then $$\displaystyle \sum_{\substack{\sigma \in \mathfrak{S}_k \\ \sigma \ k-cycle}}t^{\exc(\sigma)} = tA_{k-1}(t).$$ 
\end{lemma}

\begin{proof}
If we write each $\sigma$ in cycle form with $k$ written as the last element of the cycle, i.e. $\sigma = (\sigma_1, \sigma_2, \cdots, \sigma_{k-1}, k)$, then we obtain $\mu = \sigma_1 \sigma_2 \cdots \sigma_{k-1} \in \mathfrak{S}_{k-1}.$  This gives us a bijection between k-cycles $\sigma \in \mathfrak{S}_k$ and elements $\mu \in \mathfrak{S}_{k-1}.$ In fact, $\exc(\sigma) = \asc(\mu) + 1$ since the pair $(\sigma_{k-1}, k)$ will always form an exceedance, but $(k, \sigma_1)$ will never form an exceedence.  Hence, we have the following:

\begin{align*}
\displaystyle \sum_{\substack{\sigma \in \mathfrak{S}_k \\ \sigma \ k-cycle}}t^{\exc(\sigma)}
&= \displaystyle \sum_{\mu \in \mathfrak{S}_{k-1}}t^{1+\asc(\mu)}\\
&= t A_{k-1}(t).
\end{align*}
\end{proof}

\begin{thm} \label{Cn p thm}
Let $\lambda = (\lambda_1, \lambda_2, \cdots, \lambda_k)$ be a partition of $n$.  If $k \geq 2$, then 
\begin{equation} \label{p form Cn} \sum_{\sigma \in N_{C_n, \lambda}}t^{\inv_{\overrightarrow{C_n}}(\sigma)} = nt A_{k-1}(t) \displaystyle \prod_{i = 1}^k [\lambda_i]_t,
\end{equation} where $[n]_t = 1 + t + \cdots + t^{n-1}.$  In the case that $\lambda = (n),$ we have 
\begin{equation} \sum_{\sigma \in N_{C_n, (n)}}t^{\inv_{\overrightarrow{C_n}}(\sigma)} = nt[n-1]_t.
\end{equation} Hence the coefficient of $\frac{1}{n}p_n$ in $\omega X_{\overrightarrow{C_n}}({\bf x},t)$ is $nt[n-1]_t$ and for all other $\lambda \vdash n$, the coefficient of $z_{\lambda}^{-1}p_\lambda$ in $\omega X_{\overrightarrow{C_n}}({\bf x},t)$ is $nt  A_{k-1}(t) \prod_{i = 1}^k [\lambda_i]_t.$ 
\end{thm}

\begin{proof}
For the following proof, we will fix the labeling of $\overrightarrow{C_n}$ by $[n]$ so that $E(\overrightarrow{C_n}) = \{(i, i+1) \mid 1 \leq i < n\} \cup \{(n,1)\}.$  Note that any labeling of the vertices of $\overrightarrow{C_n}$ with $[n]$ will work the same way.

Let $\lambda = (\lambda_1, \lambda_2, \cdots, \lambda_k)$ be a partition of $n$ and let $\sigma \in \mathfrak{S}_n$ be partitioned into pieces of size $\lambda_1, \lambda_2, \cdots, \lambda_k$ so that $\sigma = \alpha_1 \alpha_2 \cdots \alpha_k,$ where $\cdot$ represents concatenation.  Then we know $\sigma \in N_{C_n, \lambda}$ if and only if each $\alpha_i$ has no $C_n$-descents and no $C_n$-isolated letters.  

For each $\alpha_i,$ we will construct a connected acyclic digraph $\overrightarrow{G_i}$ on the letters of $\alpha_i$ such that the underlying undirected graph, $G_i,$ is an induced subgraph of $C_n.$

Let $\overrightarrow{G_i}$ be the directed graph whose vertex set is the set of letters of $\alpha_i$ and whose edges have the form $(a,b)$ if $b$ precedes $a$ in $\alpha_i$ and $\{a,b\} \in E(C_n).$ Then each $\overrightarrow{G_i}$ is a connected acyclic digraph with a unique sink, which is the first letter of $\alpha_i$.  Indeed if there were another sink, then the second sink would be a $C_n$-isolated letter of $\alpha_i$.  Hence if $\lambda \neq (n),$ each underlying undirected graph, $G_i,$ is a path of length $\lambda_i$ in $C_n.$  If $\lambda = (n),$ then $G_1 = C_n$.  

For example, let $n = 9,$ $\lambda = (4,3,2)$ and $\sigma = 546389721.$  Then $\alpha_1 = 5463,$ $\alpha_2 = 897,$ and $\alpha_3 = 21.$  The corresponding acyclic digraphs as as shown below:

\begin{center}
\includegraphics[scale=0.5]{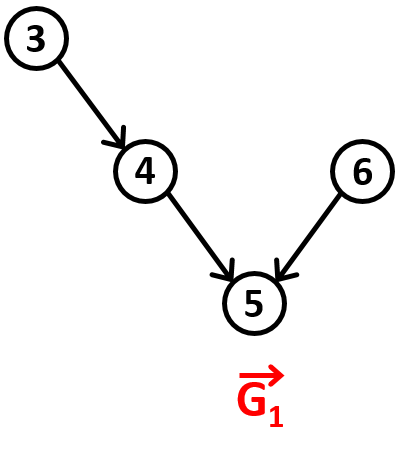} \hspace{0.3 in}
\includegraphics[scale=0.5]{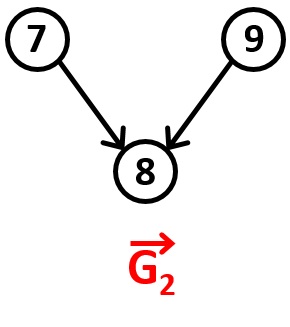}\hspace{0.3 in}
\includegraphics[scale=0.5]{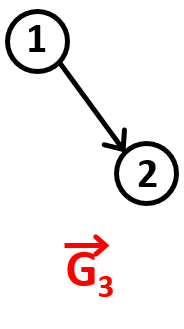}
\end{center}

We can uniquely recover $\sigma$ from the $k$-tuple $(\overrightarrow{G_1},\overrightarrow{G_2},\cdots,\overrightarrow{G_k}).$  For each vertex $x$ in each $\overrightarrow{G_i},$ let us define $\rank(x)$ as follows.  Let $x \in \overrightarrow{G_i}$ and let $V_x$ be the set of all vertices $y$ in $\overrightarrow{G_1}, \cdots, \overrightarrow{G_{i-1}}$ such that $\{x,y\}$ is an edge of $C_n.$  If $x$ is a sink of $\overrightarrow{G_i},$ then $\rank(x) = \rm{max}\{1, \substack{\rm{max}\\ y \in V_x}(\rank(y)+1)\}.$  If $x$ is not a sink of $\overrightarrow{G_i},$ then there exists a unique vertex $z$ of $\overrightarrow{G_i}$ such that $(x,z) \in E(\overrightarrow{G_i}).$  Then $\rank(x) = \rm{max}\{\rank(z)+1, \substack{\rm{max}\\ y \in V_x}(\rank(y)+1)\}.$ Then create each $\alpha_i$ by starting with all vertices of $\overrightarrow{G_i}$ of rank 1 in increasing order of their label, then all vertices of $\overrightarrow{G_i}$ of rank 2 in increasing order of their label, etc.  Then we have $\sigma = \alpha_1 \alpha_2 \cdots \alpha_k.$  Notice that for all $x \in [n],$ we have that $\rank(x) = \rank_{({C_n}, \sigma)}(x).$


Notice that the number of $\overrightarrow{C_n}$-inversions of $\alpha_i$ is the number of directed edges of $\overrightarrow{G_i}$ that are oriented in the same direction as the corresponding directed edge of $\overrightarrow{C_n}.$

\ 

\textbf{Case 1:} $\lambda = (n).$
In this case $\overrightarrow{G_1}$ is an acyclic orientation of $C_n$ with a unique sink. So we need to find the number of $\overrightarrow{C_n}$-inversions of the corresponding $\sigma,$ i.e. the number of edges of $\overrightarrow{G_1}$ that are oriented the same direction as the corresponding edge in $\overrightarrow{C_n}.$  In order to construct an acyclic orientation of $C_n$ with a unique sink (and hence a unique source), we have $n$ choices for a sink and then $n-1$ choices remaining for a source.  There are two paths from the sink to the source.  One path is oriented as in $\overrightarrow{C_n}$ and the other path is oriented opposite $\overrightarrow{C_n}.$  The number of edges of the path oriented the same direction as $\overrightarrow{C_n}$ can be $1, 2, \cdots,$ or $n-1,$ depending on the choice of the source.  So $$\displaystyle \sum_{\sigma \in N_{C_n, (n)}} t^{\inv_{\overrightarrow{C_n}}(\sigma)} = n(t + t^2 + \cdots + t^{n-1}) = nt[n-1]_t.$$

\textbf{Case 2:} $\lambda \neq (n).$
For $a,b \in \PP$ with $b \leq a,$ define a $V$-digraph $\overrightarrow{V}_{a,b}$ to be a digraph with vertex set $\{v_1, v_2, \cdots, v_a\}$ and edge set $\{(v_i, v_{i+1}) \mid 1 \leq i < b\} \cup \{(v_{i+1}, v_i) \mid b \leq i <a\}.$  We will call $v_1$ the first vertex of $\overrightarrow{V}_{a,b}$ and $v_a$ the last vertex of $\overrightarrow{V}_{a,b}.$  For $1 \leq i < a$ we say the successor of $v_i$ is $v_{i+1}.$  Let $V_{a,b}$ denote the underlying undirected graph of $\overrightarrow{V}_{a,b}.$ For all $a,b \in \PP$ with $b \leq a,$ we can see that $V_{a,b}$ is a path. For example, $\overrightarrow{V}_{4,2}$ is shown below:

\begin{center}
\includegraphics[scale=0.5]{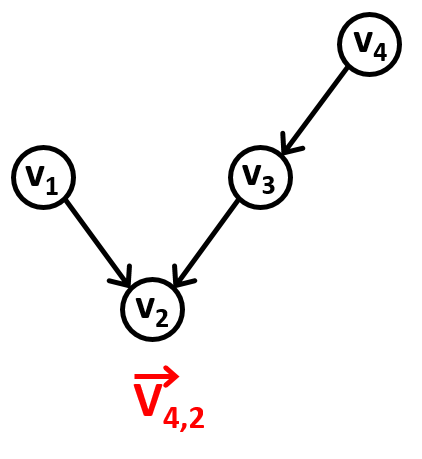}
\end{center}

Let $\lambda = (\lambda_1, \lambda_2, \cdots, \lambda_k).$  Then we will construct a bijection from $N_{C_n, \lambda}$ to the set $M_{\lambda}$ of $(k+2)$-tuples $(x, \mu,\overrightarrow{V}_{\lambda_1, b_1}, \overrightarrow{V}_{\lambda_2, b_2}, \cdots, \overrightarrow{V}_{\lambda_k, b_k}),$ where
\begin{itemize}
\item $x \in [n]$
\item $\mu \in \mathfrak{S}_k$ is a $k$-cycle
\item for each $i,$ we have $1 \leq b_i \leq \lambda_i$
\end{itemize}

Let $\sigma \in N_{C_n, \lambda}.$  Recall our earlier map from $\sigma \in N_{C_n, \lambda}$ to the $k$-tuples $(\overrightarrow{G_1}, \overrightarrow{G_2}, \cdots, \overrightarrow{G_k}).$  For each $1 \leq i \leq k,$ define $b_i$ as one more than the number of edges of $\overrightarrow{G_i}$ that match the orientation of $\overrightarrow{C_n}.$ Then $\overrightarrow{V}_{\lambda_i, b_i}$ is simply $\overrightarrow{G_i}$ without labels.  To determine $\mu = (a_1, a_2, \cdots, a_k),$ we start by letting $a_1 = j_1$ where $\overrightarrow{G_{j_1}}$ contains the vertex labeled $1.$  From the remaining $\overrightarrow{G_i},$ let $\overrightarrow{G_{j_2}}$ be the digraph with the smallest label on its sink.  Then let $a_2 = j_2.$  From the remaining $\overrightarrow{G_i},$ let $\overrightarrow{G_{j_3}}$ be the digraph with the smallest label on its sink.  Then let $a_3 = j_3.$  We continue this process until we find $a_k.$  Lastly, to determine $x,$ suppose $1$ is on the $d^{th}$ vertex of $\overrightarrow{G_i}.$  Then $x = \lambda_1 + \lambda_2 + \cdots + \lambda_{i-1} + d.$

In the other direction, suppose we have $$(x, \mu,\overrightarrow{V}_{\lambda_1, b_1}, \overrightarrow{V}_{\lambda_2, b_2}, \cdots, \overrightarrow{V}_{\lambda_k, b_k}) \in M_{\lambda}.$$  For each $1 \leq i \leq k$, we will say that the successor of the last vertex of $\overrightarrow{V}_{\lambda_i, b_i}$ is the first vertex of $\overrightarrow{V}_{\lambda_{\mu(i)}, b_{\mu(i)}}.$  

There exists unique $1 \leq i \leq k$ and $1 \leq d \leq \lambda_i$ such that $x = \lambda_1 + \lambda_2 + \cdots + \lambda_{i-1} + d.$  Place a $1$ on the $d^{th}$ vertex of $\overrightarrow{V}_{\lambda_i, b_i}.$ Then place a $2$ on its successor, and continue labeling successors in order until all $n$ vertices are labeled.  Now the labeled $\overrightarrow{V}_{\lambda_i, b_i}$ is the same as $\overrightarrow{G_i},$ so we can recover $\sigma$ as described earlier.  One can check that this is a bijection. 

Now suppose we have some $\sigma \in \mathfrak{S}_n$ that corresponds to $$(x, \mu, \overrightarrow{V}_{\lambda_1, b_1}, \overrightarrow{V}_{\lambda_2, b_2}, \cdots, \overrightarrow{V}_{\lambda_k, b_k})\in M.$$  Notice that using the bijection, the number of $\overrightarrow{C_n}$-inversions of $\alpha_i$ is equal to $b_i-1.$  One can check that the number of $\overrightarrow{C_n}$-inversions between distinct $\alpha_i$ in $\sigma$ is the same as the number of excedances of $\mu^{-1},$ because for each $i \in [k],$ there is an edge of $\overrightarrow{C_n}$ directed from the last vertex of $\overrightarrow{V}_{\lambda_i, b_i}$ to the first vertex of $\overrightarrow{V}_{\lambda_{\mu(i)}, b_{\mu(i)}}.$ Then one can see that 
$$\inv_{\overrightarrow{C_n}}(\sigma)= \exc(\mu^{-1}) + (b_1 -1) +(b_2 - 1) + \cdots (b_k -1).$$

Using \hyperref[eul lem]{Lemma \ref*{eul lem}}, we see that $$\displaystyle \sum_{\substack{{\mu \in \mathfrak{S}_k}\\{\mu \ k-cycle}}} t^{\exc(\mu^{-1})} = \displaystyle \sum_{\substack{{\mu \in \mathfrak{S}_k}\\{\mu \ k-cycle}}} t^{\exc(\mu)} = tA_{k-1}(t).$$  Now since for each $1 \leq i \leq k,$ we have $1 \leq b_i \leq \lambda_i,$ and since we have $n$ choices for $x$ in the bijection, we can see that \hyperref[p form Cn]{(\ref*{p form Cn})} is true. 
\end{proof}

\section{Symmetry} \label{Symmetry}

In this section, we define circular indifference digraphs and show that they have symmetric chromatic quasisymmetric functions.  For $a, b \in [n],$ we define the \textit{circular interval} $[a,b]$ of $[n]$ as
$$ [a,b] := \begin{cases} \{a, a+1, a+2, ..., b\} &\mbox{ if } a\leq b \\  \{a, a+1, a+2, \cdots, n, 1, 2, \cdots, b\} &\mbox{ if } a>b \end{cases}.$$

\begin{definition} We call a digraph, $\overrightarrow{G} = ([n],E)$, a \textit{circular indifference digraph} if there exists a collection of circular intervals, $I,$ of $[n]$ such that $E =\{(i,j)\mid [i,j] \mbox{ is contained in a circular interval of }I \}.$
\end{definition}

\begin{example} Suppose we have the set of circular intervals $I = \{[1,3], [2,4], [4,5], [5,1]\}.$  Then the corresponding circular indifference digraph is shown below.

\centering \includegraphics[scale=0.5]{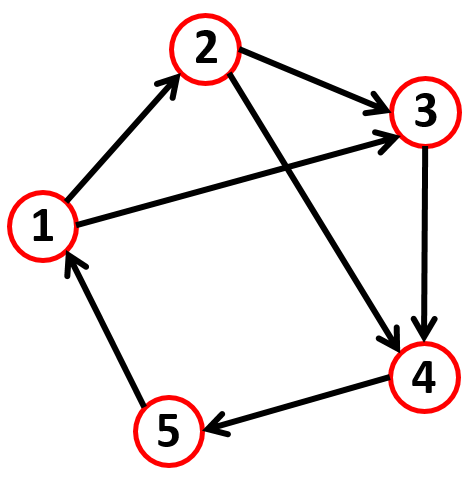}
\end{example}

The underlying undirected graphs of these circular indifference digraphs are the \textit{circular indifference graphs} defined by Stanley in \cite{CSF}.  We discuss circular indifference graphs and their relation to other well-known classes of graphs in \hyperref[Graphs]{Appendix \ref*{Graphs}}.  

In \cite[Theorem 4.5]{CQSF}, Shareshian and Wachs show that $X_G({\bf x},t)$ is symmetric if $G$ is a natural unit interval graph.  As discussed in \hyperref[Graphs]{Appendix \ref*{Graphs}}, when natural unit interval graphs are viewed as digraphs, they are acyclic circular indifference digraphs.  Next we extend the symmetry result of Shareshian and Wachs to all circular indifference digraphs.  Our proof of symmetry is similar to that of Shareshian and Wachs.  First we need the following lemmas.

Let us define five digraphs we will need for the next lemma.  
\begin{itemize}
\item $\overrightarrow{K_{12}} = (\{a,b,c\}, \{(b,a), (b,c)\}).$  
\item $\overrightarrow{K_{21}}=(\{a, b, c\}, \{(a,b), (c,b)\}).$  
\item $\overleftarrow{\overrightarrow{K_{12}}} = (\{a, b, c\}, \{(a,b), (b,a), (b,c)\}).$  
\item $\overleftarrow{\overrightarrow{K_{21}}} = (\{a, b, c\}, \{(a,b), (b,a), (c,b)\}).$  
\item $\overleftarrow{\overrightarrow{P_3}} = (\{a, b, c\}, \{(a,b), (b,a), (b,c), (c,b)\}).$  
\end{itemize}
Below we see all five digraphs.

\begin{center}
\includegraphics[scale=0.5]{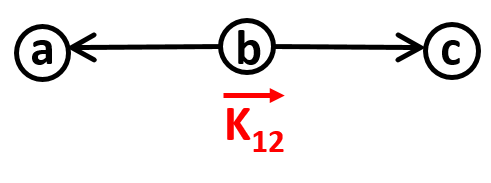} \hspace{0.5 in} \includegraphics[scale=0.5]{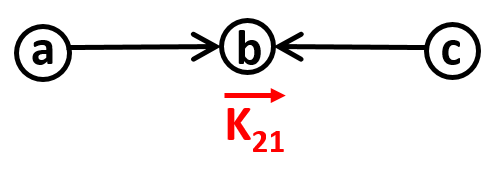}

\includegraphics[scale=0.5]{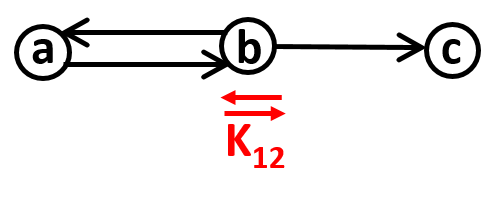} \hspace{0.5 in} \includegraphics[scale=0.5]{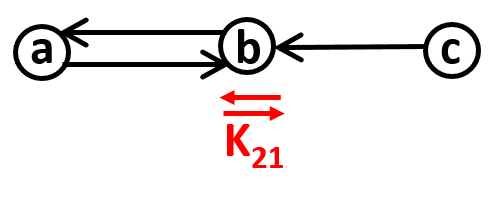}

\includegraphics[scale=0.5]{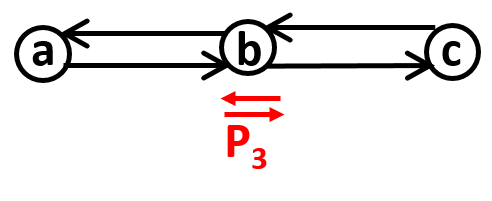}
\end{center}

\begin{lemma} \label{claw-free lemma}
Let $\overrightarrow{G}$ be a digraph that has no induced subdigraphs isomorphic to $\overrightarrow{K_{12}}, \overrightarrow{K_{21}},\overleftarrow{\overrightarrow{K_{12}}}, \overleftarrow{\overrightarrow{K_{21}}}$ or $\overleftarrow{\overrightarrow{P_3}}.$  Then the underlying undirected graph, $G,$ is claw-free, i.e. $G$ does not contain an induced subgraph isomorphic to $K_{31}.$
\end{lemma}

\begin{proof}
Let $\overrightarrow{G}$ be a digraph whose underlying undirected graph is the claw, $K_{31}.$  It is not difficult to see that $\overrightarrow{G}$ must have an induced subdigraph isomoprhic to one of the five digraphs listed.  But this means that any digraph that contains an induced claw subgraph must contain a forbidden subdigraph.
\end{proof}

For the next lemma, we need a few definitions.  We say that a digraph, $\overrightarrow{G},$ is \textit{connected} if its underlying undirected graph, $G,$ is connected. We say that a subdigraph, $\overrightarrow{H},$ is a \textit{connected component} of $\overrightarrow{G}$ if $H$ is a connected component of $G.$   As in \hyperref[dir path ex]{Example \ref*{dir path ex}}, we say a digraph, $\overrightarrow{G}=(V,E)$ is a \textit{directed path} if its vertex set is $V = \{v_1, v_2, \cdots, v_n\}$ and its edge set is $E = \{(v_i, v_{i+1}) \mid 1 \leq i <n\}.$  As defined in \hyperref[dir cyc ex]{Example \ref*{dir cyc ex}}, we say a digraph, $\overrightarrow{G} = (V,E),$ is a \textit{directed cycle} if its vertex set is $V = \{v_1, v_2, \cdots, v_n\}$ and its edge set is $E = \{(v_i, v_{i+1}) \mid 1 \leq i <n\} \cup \{(v_n, v_1)\}.$

\begin{lemma} \label{Ka lemma}
Let $\overrightarrow{G}$ be a digraph that has no induced subdigraphs isomorphic to $\overrightarrow{K_{12}}, \overrightarrow{K_{21}},\overleftarrow{\overrightarrow{K_{12}}}, \overleftarrow{\overrightarrow{K_{21}}}$ or $\overleftarrow{\overrightarrow{P_3}}.$  Let $\kappa$ be a proper coloring of $\overrightarrow{G}$.  For $a \in \PP,$ define $\overrightarrow{G_{\kappa,a}}$ as the induced subdigraph of $\overrightarrow{G}$ of all vertices colored by $a$ or $a+1$.  Then each connected component of $\overrightarrow{G_{\kappa, a}}$ is either a directed cycle with an even number of vertices or a directed path.
\end{lemma}

\begin{proof}
Let $G_{\kappa, a}$ be the underlying undirected graph of $\overrightarrow{G_{\kappa, a}}.$  First note that $G_{\kappa, a}$ cannot have any cycles of odd length, because then two vertices with the same color would be adjacent, which contradicts the fact that $\kappa$ is a proper coloring.  

We can also see that $G_{\kappa, a}$ cannot have any vertex adjacent to more than two other vertices.  Indeed, suppose vertex $v$ were adjacent to vertices $w_1,$ $w_2,$ and $w_3$ in $G_{\kappa, a},$ as in the following figure:
 
\begin{center}\includegraphics[scale=0.4]{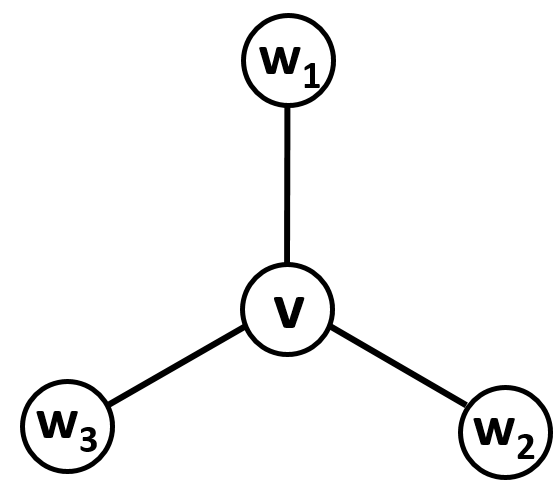}\end{center}
Since $G_{\kappa, a}$ has no 3-cycles, $w_1,$ $w_2,$ and $w_3$ have no edges between them. Then we see that $G_{\kappa, a}$ contains a claw as an induced subgraph, but this contradictions \hyperref[claw-free lemma]{Lemma \ref*{claw-free lemma}}

Then since every vertex has degree at most 2, every connected component of $G_{\kappa, a}$ must be either a path or a cycle of even length. Since every induced subdigraph of $\overrightarrow{G}$ must avoid $\overrightarrow{K_{12}},$ $\overrightarrow{K_{21}},$ $\overleftarrow{\overrightarrow{K_{12}}},$ $\overleftarrow{\overrightarrow{K_{21}}},$ and $\overleftarrow{\overrightarrow{P_3}},$ the only possible connected components are the ones listed in the lemma.
\end{proof}

\begin{thm} \label{sym thm}
Let $\overrightarrow{G}$ be a digraph that has no induced subdigraphs isomorphic to $\overrightarrow{K_{12}}, \overrightarrow{K_{21}},\overleftarrow{\overrightarrow{K_{12}}}, \overleftarrow{\overrightarrow{K_{21}}},$ or $\overleftarrow{\overrightarrow{P_3}}.$  Then $X_{\overrightarrow{G}}({\bf x},t)$ is symmetric.
\end{thm}

\begin{proof}
By \hyperref[disj]{Proposition \ref*{disj}}, we can assume without loss of generality that $\overrightarrow{G}$ is connected.

We will construct an involution, $\phi_a$, for each $a \in \PP$ on the set of proper colorings of $\overrightarrow{G}$ that switches the number of occurrences of the color $a$ and the number of occurrences the color $a+1$, leaves the number of occurrences of all other colors the same, and does not change the number of ascents of the coloring.  This will then prove the theorem.

So let $\kappa$ be a proper coloring of $\overrightarrow{G}$ and let $\overrightarrow{G_{\kappa, a}}$ be the induced subdigraph of $\overrightarrow{G}$ containing only the vertices colored by $a$ and $a+1$.  By \hyperref[Ka lemma]{Lemma \ref*{Ka lemma}}, each component of $\overrightarrow{G_{\kappa, a}}$ is a directed path or a directed cycle of even length.

Let $\phi_a(\kappa)$ be the the coloring of $\overrightarrow{G}$ obtained from $\kappa$ by replacing each occurence of $a$ with $a+1$ and replacing each $a+1$ with $a$ in the components of $\overrightarrow{G_{\kappa, a}}$ that are paths with an odd number of vertices.  For the other components of $\overrightarrow{G_{\kappa, a}}$	(paths and cycles with an even number of vertices), the colors of $\phi_a(\kappa)$ are the same as those of $\kappa.$

Note that in a path of with an odd number of vertices in $\overrightarrow{G_{\kappa, a}},$ exactly half of the edges are ascents of $\kappa.$  Hence, if we change all $a$'s to $a+1$'s and vice versa, we will change all ascents to descents and vice versa, but the number of ascents of $\kappa$ is preserved.  It is then easy to see that $\phi_a$ is an involution that meets the desired conditions and hence the theorem is proven.

\end{proof}

\begin{lemma} \label{induced}
Circular indifference digraphs do not have any induced subdigraphs isomorphic to  $\overrightarrow{K_{12}}, \overrightarrow{K_{21}},\overleftarrow{\overrightarrow{K_{12}}}, \overleftarrow{\overrightarrow{K_{21}}},$ or $\overleftarrow{\overrightarrow{P_3}}.$
\end{lemma}

\begin{proof}
Let $\overrightarrow{G}$ be a circular indifference digraph arising from a set of circular intervals, $I,$ on $[n],$ and suppose $\overrightarrow{G}$ contains an induced subdigraph, $\overrightarrow{H}$, isomorphic to $\overrightarrow{K_{12}}.$  Suppose $\overrightarrow{H}$ has vertex set $\{a, b, c\}$ and edge set $\{(b,a), (b,c)\}.$ Then the circular intervals $[b,a]$ and $[b,c]$ are both contained in circular intervals of $I.$  But then either $[b,a] \subset [b,c]$ and hence $[a,c] \subset [b,c],$ which is contained in a circular interval of $I,$ or $[b,c] \subset [b,a]$ and hence $[c,a] \subset [b,a],$ which is contained in a circular interval of $I.$ Either way there is an edge between $a$ and $c$ in $\overrightarrow{G},$ which is a contradiction.  Similar arguments show that $\overrightarrow{G}$ cannot contain any induced subdigraphs isomorphic to $\overrightarrow{K_{21}},$ $\overleftarrow{\overrightarrow{K_{12}}}, \overleftarrow{\overrightarrow{K_{21}}},$ or $\overleftarrow{\overrightarrow{P_3}}.$
\end{proof}

\begin{cor}
Let $\overrightarrow{G}$ be a digraph such that all connected components of $\overrightarrow{G}$ are circular indifference digraphs.  Then $X_{\overrightarrow{G}}({\bf x},t)$ is symmetric.
\end{cor}

\begin{proof}
Combine \hyperref[induced]{Lemma \ref*{induced}} with \hyperref[sym thm]{Theorem \ref*{sym thm}}.
\end{proof}

It turns out that the class of digraphs from \hyperref[sym thm]{Theorem \ref*{sym thm}} are not the only digraphs with symmetric chromatic quasisymmetric functions.  In fact, the directed cycle with one edge directed backwards, i.e. the digraph with vertex set $V = \{v_1, v_2, \cdots, v_n\}$ and edge set $E = \{(v_i, v_{i+1})\mid 1 \leq i <n\} \cup \{(v_1, v_n)\},$ is symmetric (see \cite[Exercise 2.84]{hopf}), but it contains induced subdigraphs isomorphic to both $\overrightarrow{K_{21}}$ and $\overrightarrow{K_{12}}$. See \cite{me2} for a proof of symmetry and further results. 

\section{Expansion in the elementary symmetric function basis} \label{e-basis}

In this section, we provide some evidence for \hyperref[e-pos conj]{Conjecture \ref*{e-pos conj}}.  Below we take a look at the simplest example of a circular indifference digraph that is not acyclic, namely the directed cycle, $\overrightarrow{C_n}$.

\begin{thm} \label{cycle exp} Let $\overrightarrow{C_n}$ be the directed cycle of length $n$.  Then
\begin{equation}\label{Cn eqn} \displaystyle \sum_{n \geq 2} X_{\overrightarrow{C_n}}({\bf x},t)z^n = \frac{t \displaystyle \sum_{k \geq 2}k[k-1]_t e_k({\bf x}) z^k}{1-t \displaystyle \sum_{k \geq 2} [k-1]_t e_k({\bf x}) z^k},
\end{equation} where $[n]_t = 1 + t + t^2 + \cdots + t^{n-1}.$
\end{thm}

\begin{proof}
This proof is based on the proof of Proposition 5.4 in \cite{CSF}, which uses the Transfer-Matrix Method \cite[Section 4.7]{EC1}.  Let us first look at all proper colorings of all $\overrightarrow{C_d}$ for $d \geq 2$ using only $n$ colors.  We will view these colorings as closed walks of length $d$ on the colors $[n]$.  Each time we take a step, the colors either increase (and we need to count an ascent) or decrease.  Let $\overrightarrow{G} = ([n],E)$ be the digraph with $E = \{(i,j) \mid i \neq j\}.$  Let us label each edge $(i,j)$ with $tx_i$ if $i<j$  or with $x_i$ if $i>j.$  The weighted adjacency matrix of $\overrightarrow{G}$ is given by
$$ A = \begin{bmatrix}
    0 & tx_{2} & tx_{3} & \dots  & tx_{n} \\
    x_{1} & 0 & tx_{3} & \dots  & tx_{n} \\
    x_{1} & x_{2} & 0 & \dots & tx_{n}\\
    \vdots & \vdots & \vdots & \ddots & \vdots \\ 
    x_{1} & x_{2} & x_{3} & \dots  & 0
\end{bmatrix}.$$

Let $Q(z) = \dtr(I - zA)$.  By \cite[Corollary 4.7.3]{EC1}, we know that $$\displaystyle \sum_{d \geq 2} X_{\overrightarrow{C}_d}({\bf x},t)|_{x_1, x_2,...x_n}z^d = \frac{-zQ'(z)}{Q(z)}.$$  So we need to compute $$Q(z)= \dtr(I - zA) =
 \dtr\begin{bmatrix}
    1 & -tx_{2}z & -tx_{3}z & \dots  & -tx_{n}z \\
    -x_{1}z & 1 & -tx_{3}z & \dots  & -tx_{n}z \\
    -x_{1}z & -x_{2}z & 1 & \dots & -tx_{n}z\\
    \vdots & \vdots & \vdots & \ddots & \vdots \\
    -x_{1}z & -x_{2}z & -x_{3}z & \dots  & 1
\end{bmatrix}.$$

First let us describe the notation we will use. For $\sigma \in \mathfrak{S}_n,$ $\FIX(\sigma) = \{i \mid \sigma(i) = i\},$ $\exc(\sigma) = |\{i \mid \sigma(i)>i\}|,$ and $\sgn(\sigma)$ is the usual sign function on permutations.  We say $\sigma$ is a derangement if $\FIX(\sigma) = \emptyset.$  Also for any $S \subseteq [n],$ define ${\bf x}_S = \prod_{i \in S} x_i.$

Then 
\begin{align*}
Q(z) &= \displaystyle \sum_{\sigma \in \S_n} \sgn(\sigma)(-1)^{n-|\FIX(\sigma)|} z^{n-|\FIX(\sigma)|} t^{\exc(\sigma)} {\bf x}_{[n]\backslash \FIX(\sigma)} \\
&= 1+ \displaystyle \sum_{k=1}^n (-1)^k z^k \displaystyle \sum_{\substack{S \subseteq [n] \\ |S| = k}} {\bf x}_S \displaystyle \sum_{\substack{\sigma \in \S_n \\ \FIX(\sigma) = [n]\backslash S}} \sgn(\sigma) t^{\exc(\sigma)}\\
&= 1 + \displaystyle \sum_{k=1}^n (-1)^k z^k \displaystyle \sum_{\substack{S \subseteq [n] \\ |S| = k}} {\bf x}_S \displaystyle \sum_{\substack{\sigma \in \S_k \\ \sigma \ \rm is \ \rm a \ \rm derangement}} \sgn(\sigma) t^{\exc(\sigma)}\\
&= 1 + \displaystyle \sum_{k=1}^n (-1)^k z^k e_k(x_1, x_2,...,x_n) \displaystyle \sum_{\substack{\sigma \in \S_k \\ \sigma \ \rm is \ \rm a \ \rm derangement}} \sgn(\sigma) t^{\exc(\sigma)}.
\end{align*}

By \cite[Corollary 5.11]{MR}, we have that $$\displaystyle \sum_{\substack{\sigma \in \S_k \\ \sigma\ \rm is\ a\ derangement}} \sgn(\sigma) t^{\exc(\sigma)} = (-1)^{k+1}t[k-1]_t.$$

Hence
\begin{align*}
Q(z) &= 1 + \displaystyle \sum_{k=1}^n (-1)^k z^k e_k(x_1, x_2,...,x_n) (-1)^{k+1}t[k-1]_t \\
&= 1 - t \displaystyle \sum_{k \geq 2} e_k(x_1, x_2,...,x_n)[k-1]_tz^k.
\end{align*}

Then we have
\begin{align*}
-zQ'(z) &= -z(-t \displaystyle \sum_{k \geq 2} e_k(x_1, x_2,...,x_n)k[k-1]_t z^{k-1})\\
 &= t \displaystyle \sum_{k \geq 2} e_k(x_1, x_2,...,x_n)k[k-1]_t z^{k}.
\end{align*}

Letting $n$ go to infinity gives us our result.
\end{proof}

\begin{cor} \label{e-uni}
For each $n \geq 2,$ \begin{equation}\label{Cn form} X_{\overrightarrow{C_n}}({\bf x},t) = \sum_{\lambda \vdash n} e_{\lambda}({\bf x}) \sum_{\mu: \lambda(\mu) = \lambda} \mu_1 t[\mu_1-1]_tt[\mu_2-1]_t\cdots t[\mu_{l(\lambda)}-1]_t,
\end{equation}
where $\mu = (\mu_1, \mu_2, \cdots, \mu_{l(\lambda)})$ is a composition of $n,$ $l(\lambda)$ is the length of $\lambda,$ and $\lambda(\mu) = \lambda$ means that when the parts of $\mu$ are written in descreasing order, we get the partition $\lambda.$

Consequently, $X_{\overrightarrow{C_n}}({\bf x},t)$ is a palindromic, $e$-positive and $e$-unimodal polynomial in $t$.
\end{cor}

\begin{proof} 
We can rewrite \hyperref[Cn eqn]{(\ref*{Cn eqn})} as
$$\displaystyle \sum_{n \geq 2} X_{\overrightarrow{C_n}}({\bf x},t)z^n = (t \displaystyle \sum_{k \geq 2}k[k-1]_t e_k z^k)(\displaystyle \sum_{m \geq 0}(t \displaystyle \sum_{l \geq 2} [l-1]_t e_l z^l)^m),$$
and from here we can see \hyperref[Cn form]{(\ref*{Cn form})}.

From \hyperref[Cn form]{(\ref*{Cn form})}, we see that the coefficient of each $e_\lambda$ is a sum of products of positive, palindromic, unimodal polynomials in $t$ with centers of symmetry $\frac{\lambda_i}{2}$ for each $i.$  By \cite[Proposition 1]{Stanlog}, we know that any product of these polynomials will result in another positive, palindromic, unimodal polynomial in $t$ with center of symmetry equal to the sum of the centers of symmetry of each polynomial in the product.  So once we multiply this out, each term in the sum will be a positive, palindromic, unimodal polynomial in $t$ with center of symmetry $\frac{n}{2}$ (see \cite[Proposition B.3]{CQSF}). Since the sum of positive, palindromic, unimodal polynomials in $t$ with center of symmetry $\frac{n}{2}$ is a positive, palindromic, unimodal polynomial in $t$ with center of symmetry $\frac{n}{2},$ we have our result. 
\end{proof}

\begin{prop}\label{Kn prop}
Let $\overrightarrow{G} = ([n],E)$ be a digraph whose underlying undirected graph, $G$, is the complete graph, $K_n$.  Then $$X_{\overrightarrow{G}}({\bf x},t) = p(t)e_n({\bf x}),$$ where $$p(t) = \sum_{\sigma \in \mathfrak{S}_n} t^{inv_{\overrightarrow{G}}(\sigma)}.$$  As a result, $X_{\overrightarrow{G}}({\bf x},t)$ is symmetric and $e$-positive.
\end{prop}

\begin{proof}
Since every vertex of $\overrightarrow{G}$ is adjacent to every other vertex, we can see that for every $\sigma = \sigma_1 \sigma_2 \cdots \sigma_n \in \mathfrak{S}_n$ we have $\rank_{(G, \sigma)}(\sigma_i) = i$ for each $i \in [n],$ so $\sigma$ contains no $G$-descents.  Taking $\omega$ of both sides of our $F$-basis expansion from \hyperref[F-basis thm]{Theorem \ref*{F-basis thm}} and the fact that $\omega F_{n, \emptyset}({\bf x}) = e_n({\bf x})$ gives us our result. 
\end{proof}

In \cite{CQSF}, Shareshian and Wachs introduce a class of graphs, which they call $G_{n,r}$ graphs, where $n \in \PP$ and $1 \leq r \leq n$.  The vertices of $G_{n,r}$ are labeled by $[n]$ and for $1 \leq i < j \leq n$ there is an edge between $i$ and $j$ if $0 < j - i < r.$  For example, $G_{n,1}$ is the graph on $n$ vertices with no edges, $G_{n,2}$ is the labeled path on $n$ elements, and $G_{n, n}$ is the complete graph on $n$ elements.  Shareshian and Wachs proved that their $e$-positivity conjecture holds for all $G_{n,r}$ when $r = 1, 2, n-2, n-1, n$ and they tested by computer all $G_{n,r}$ for $n \leq 8.$  Hence if these graphs are turned into digraphs by orienting their edges from smaller label to larger label, our $e$-positivity conjecture holds for the same graphs.

We present a circular analog of these graphs, which we will call $\overrightarrow{G}^c_{n,r},$ where $n \in \PP$ and $1 \leq r \leq n.$  We define $\overrightarrow{G}^c_{n,r} = ([n],E)$, where $E = \{(i,j) \mid 0< j-i \pmod n < r\}.$  In other words, $\overrightarrow{G}^c_{n,r}$ is the circular indifference digraph on $[n]$ arising from the set of circular intervals $$I = \{[i,i+r-1] \mid 1 \leq i \leq n-r+1\} \cup \{[i, i+r-1-n] \mid n-r+2 \leq i \leq n\}.$$  For example, $\overrightarrow{G}^c_{n,1} = ([n], \emptyset),$ $\overrightarrow{G}^c_{n,2}$ is the directed cycle, $\overrightarrow{C_n},$ and $\overrightarrow{G}^c_{n,n} = ([n],E),$ where $E = \{(i,j) \mid i \neq j\}$. \hyperref[e-uni]{Corollary \ref*{e-uni}} proves that our $e$-positivity conjecture (\hyperref[e-pos conj]{Conjecture \ref*{e-pos conj}}) holds for $\overrightarrow{G}^c_{n,2}$.  It is easy to see that our $e$-positivity conjecture holds for $\overrightarrow{G}^c_{n,1}.$  Below we show that our conjecture holds for $\overrightarrow{G}^c_{n,r}$ when $r = n-1, n.$ We used a computer to test our conjecture for all other $\overrightarrow{G}^c_{n,r}$ for $n \leq 8.$

\begin{thm}\label{Gnn thm}
For all $n \geq 1$ we have \begin{equation}\label{Gnn} X_{\overrightarrow{G}^c_{n,n}}({\bf x},t) = n! e_n({\bf x}) t^{(\substack{{n} \\ {2}})},
\end{equation} and \begin{equation}\label{Gnn-1}X_{\overrightarrow{G}^c_{n,n-1}}({\bf x},t) = n e_n({\bf x}) t^{(\substack{{n} \\ {2}})-n+1} A_{n-1}(t).
\end{equation}
\end{thm}

\begin{proof}
First we will prove \hyperref[Gnn]{(\ref*{Gnn})}. From \hyperref[Kn prop]{Proposition \ref*{Kn prop}}, we see that $$X_{\overrightarrow{G}^c_{n,n}}({\bf x},t) = e_n({\bf x}) \sum_{\sigma \in \mathfrak{S}_n} t^{\inv_{\overrightarrow{G}^c_{n,n}}(\sigma)}.$$  But since every pair $(i,j)$ is an edge, we see that for every $\sigma \in \mathfrak{S}_n$, $\inv_{\overrightarrow{G}^c_{n,n}}(\sigma) = (\substack{{n}\\{2}}$).  Combining this with the fact that $|\mathfrak{S}_n| = n!,$ we have our formula for $X_{\overrightarrow{G}^c_{n,n}}({\bf x},t).$

Now let us prove \hyperref[Gnn-1]{(\ref*{Gnn-1})}. The graph $\overrightarrow{G}^c_{n,n-1} = ([n],E)$ has edge set $E = \{(i,j) \mid i-j \neq 0, 1, 1-n\}.$   The set, $E$, can be divided into two types: the exterior edges, $$\{(1,2), (2,3), \cdots, (n-1, n), (n,1)\},$$ which form the directed cycle, $\overrightarrow{C_n},$ and the interior edges, which are the remaining edges.  Note that the interior edges are two-way edges; that is, if $(a,b)$ is an interior edge, then so is $(b,a)$.  Below is $\overrightarrow{G}^c_{4,3}$ where the exterior edges are solid black arrows and the interior edges are dotted red arrows.

\begin{center}
\includegraphics[scale=0.5]{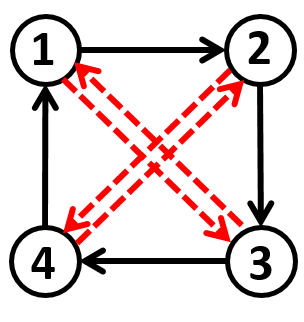}
\end{center}

Again we will use \hyperref[Kn prop]{Proposition \ref*{Kn prop}}, so we know 
$$X_{\overrightarrow{G}^c_{n,n-1}}({\bf x},t) = e_n({\bf x}) \sum_{\sigma \in \mathfrak{S}_n} t^{\inv_{\overrightarrow{G}^c_{n,n-1}}(\sigma)}.$$  Notice that for each $\sigma \in \mathfrak{S}_n,$ we have $(\substack{{n}\\{2}}) - n$ ${\overrightarrow{G}^c_{n,n-1}}$-inversions coming from the interior edges.  In order to count the ${\overrightarrow{G}^c_{n,n-1}}$-inversions from the exterior edges, recall that the exterior edges form the directed cycle, $\overrightarrow{C_n}$, so we need to find $\sum_{\sigma \in \mathfrak{S}_n} t^{\inv_{\overrightarrow{C_n}}(\sigma)}.$  Setting $\lambda = 1^n$ in \hyperref[p form Cn]{(\ref*{p form Cn})} gives us that $\sum_{\sigma \in \mathfrak{S}_n} t^{\inv_{\overrightarrow{C_n}}(\sigma)} = ntA_{n-1}(t).$  Combining all this gives us \hyperref[Gnn-1]{(\ref*{Gnn-1})}.
\end{proof}

The following theorem can be easily proven using the same proof technique as Stanley used to prove \cite[Theorem 3.3]{CSF}.

\begin{thm} \label{e sum}
Let $\overrightarrow{G}$ be a digraph on $n$ vertices such that $X_{\overrightarrow{G}}({\bf x},t)$ is symmetric.  Suppose we have the expansion $X_{\overrightarrow{G}}({\bf x},t) = \displaystyle \sum_{\lambda \vdash n} c_{\lambda}(t)e_{\lambda}({\bf x}).$  Then 
\begin{equation} \label{e sum eqn} \displaystyle \sum_{\substack{\lambda \vdash n \\ l(\lambda) = k}} c_{\lambda}(t) = \displaystyle \sum_{G_{\bar{a}} \in AO_k(G)} t^{\asc_{\overrightarrow{G}}(G_{\bar{a}})},
\end{equation} where $G$ is the underlying undirected graph of $\overrightarrow{G}$, $AO_k(G)$ is the set of acyclic orientations of $G$ with exactly $k$ sinks and $\asc_{\overrightarrow{G}}(G_{\bar{a}})$ is the number of directed edges of $\overrightarrow{G}$ that are oriented as in $G_{\bar{a}}.$
\end{thm}

\begin{cor}
Let $\overrightarrow{G}$ be a digraph on $n$ vertices such that $X_{\overrightarrow{G}}({\bf x},t)$ is symmetric.  Then $$c_{(n)}(t) = \sum_{G_{\bar{a}}\in AO_1(G)} t^{asc_{\overrightarrow{G}}(G_{\bar{a}})}.$$  
\end{cor}

For the directed path and the directed cycle, we can refine \hyperref[e sum eqn]{(\ref*{e sum eqn})} by giving a combinatorial interpretation of each $c_\lambda(t)$ in terms of acyclic orientations.  We already know that the $c_\lambda(t)$ have positive coefficients by the formula given in \cite[Theorem 7.2]{Eul} and by \hyperref[e-uni]{Corollary \ref*{e-uni}}, but perhaps these interpretations can be generalized to show $e$-positivity for a larger class of graphs.

For the next proposition, let $\overrightarrow{C_n} = ([n],E)$ denote the directed cycle, where $E = \{(i, i+1) \mid 1 \leq i <n\} \cup \{(n,1)\},$ and let $C_n$ denote its underlying undirected graph.  For an acyclic orientation of $C_n$, denoted $G_{\bar {a}},$ we say that $i$ and $j$ are \textit{consecutive sinks} of $G_{\bar{a}}$ if $i$ and $j$ are both sinks of $G_{\bar{a}}$ and there are no other sinks in the circular interval $[i, j].$

\begin{prop}\label{Cn acyc}
Let $X_{\overrightarrow{C_n}}({\bf x},t) = \displaystyle \sum_{\lambda \vdash n} c_{\lambda}(t) e_{\lambda}({\bf x}).$  Then $$c_{\lambda}(t) = \displaystyle \sum_{G_{\bar{a}} \in AO_{\lambda}(C_n)} t^{\asc_{\overrightarrow{C_n}}(G_{\bar{a}})},$$ where $AO_{\lambda}(C_n)$ is the set of all acyclic orientations $G_{\bar{a}}$ of $C_n$ such that the number of vertices between consecutive sinks of $G_{\bar{a}}$ is $\lambda_1 - 1,$ $\lambda_2 - 1,$...,$\lambda_k - 1$ in any order and $\asc_{\overrightarrow{C_n}}(G_{\bar{a}})$ is the number of directed edges of $\overrightarrow{C_n}$ that are oriented as in $G_{\bar{a}}.$
\end{prop}

\begin{example}
In the acyclic orientation of $C_9$ shown below, there are 3 vertices between sinks {\bf 2} and {\bf 6}, 1 vertex between sinks {\bf 6} and {\bf 8} and 2 vertices between sinks {\bf 8} and {\bf 2}, so this corresponds to $e_{432}.$  There are 3 edges that match the original cyclic orientation of $\overrightarrow{C_9},$ shown by the dotted red arrows, hence this acyclic orientation corresponds to $t^3 e_{432}.$
\begin{center}
\includegraphics[scale=0.5]{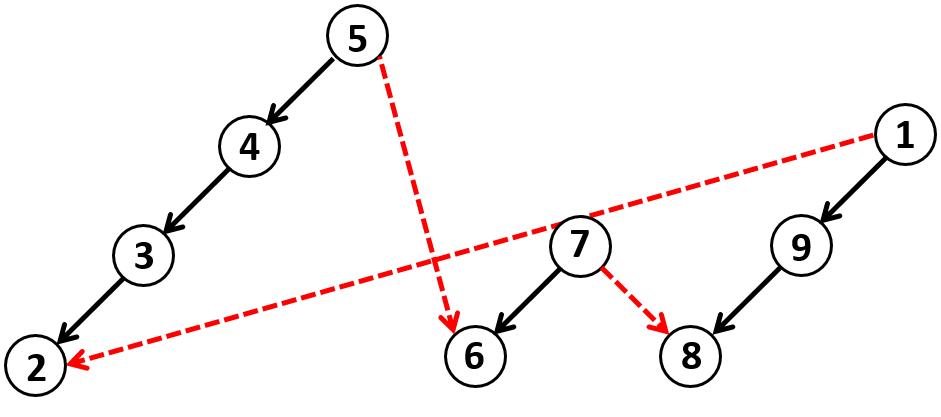}
\end{center}
\end{example}

\begin{proof}
By \hyperref[e-uni]{Corollary \ref*{e-uni}} \begin{equation}\label{form}
c_\lambda = \sum_{\mu: \lambda(\mu) = \lambda} \mu_1 t[\mu_1-1]_t t[\mu_2-1]_t \cdots t[\mu_{l(\lambda)}-1]_t,
\end{equation}
where $\mu = (\mu_1, \mu_2, \cdots, \mu_{l(\lambda)})$ is a composition of $n,$ $l(\lambda)$ is the length of $\lambda,$ and $\lambda(\mu) = \lambda$ means that when the parts of $\mu$ are written in decreasing order, this is the partition $\lambda.$

It follows from this that $c_\lambda = 0$ if any of the parts of $\lambda$ = 1.  We also have $AO_{\lambda}(C_n)$ is empty in that case, which means that the result holds in that case.  We can now assume that $\lambda$ has no parts of size 1.

For $a, b \in \PP$ with $1\leq b <a,$ define a \textit{mountain}, $\overrightarrow{M}_{a,b} = (V, E),$ as a digraph on $a$ vertices $V = \{v_1, v_2, \cdots, v_a\}$ with edge set $E = \{(v_i, v_{i-1}) \mid 1 < i \leq a-b\} \cup \{(v_i, v_{i+1}) \mid a-b \leq i <a\}.$  We will say $v_1$ is the first vertex of the mountain and $v_a$ is the last.  For each $i = 1, 2, \cdots, a-1,$ we say that $v_{i+1}$ is the successor of $v_i$ and $v_i$ is the predecessor of $v_{i+1}.$  Below, we show $\overrightarrow{M}_{5,3}.$

\begin{center}
\includegraphics[scale=0.5]{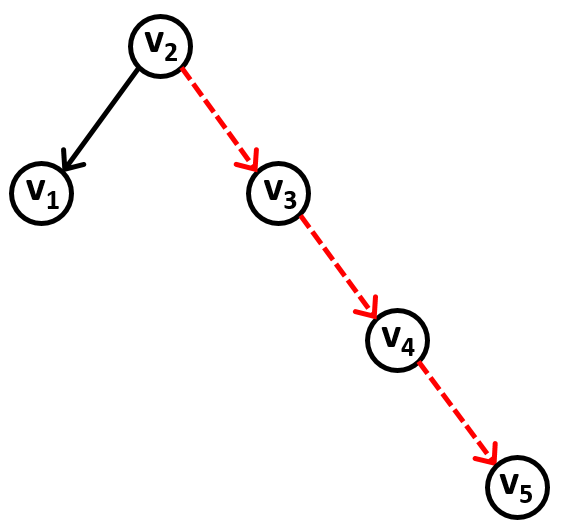}
\end{center}

We can obtain an acyclic orientation of $C_n$ from each term of the inner sum of \hyperref[form]{(\ref*{form})} as follows.  For each $1 \leq i \leq , l(\lambda),$ suppose we choose the $t^{j_i}$ term from the $t[\mu_i -1]_t$ factor.  From this choice of $j_i$'s, we can create a sequence of mountains, $\overrightarrow{M}_{\mu_1+1, j_1}, \overrightarrow{M}_{\mu_2+1, j_2}, \cdots \overrightarrow{M}_{\mu_{l(\lambda)}+1, j_{l(\lambda)}},$ on pairwise disjoint vertex sets.  Then we attach the mountains by identifying the last vertex of $\overrightarrow{M}_{\mu_i+1, j_i}$ with the first vertex of $\overrightarrow{M}_{\mu_{i+1}+1, j_{i+1}}$ for $1 \leq i <l(\lambda)$ and by identifying the last vertex of $\overrightarrow{M}_{\mu_{l(\lambda)}+1, j_{l(\lambda)}}$ with the first vertex of $\overrightarrow{M}_{\mu_1+1, j_1}.$  

We will place the label $\bf 1$ on one of the vertices, $v$, from $\overrightarrow{M}_{\mu_1+1, j_1},$ excluding the last vertex, so the $\mu_1$ factor in \hyperref[form]{(\ref*{form})} is for our $\mu_1$ choices. We label the successor of $v$ with $\bf 2$ and continue labeling successors in order until we reach the predecessor of $v$.

It should be clear that we get a unique acyclic orientation in this manner and that every acyclic orientation can be built with this method.  This proves our proposition.
\end{proof}

For the following proposition, let $\overrightarrow{P_n} = ([n],E)$ denote the directed path, where $E = \{(i, i+1)\mid 1 \leq i <n\},$ and let $P_n$ denote the underlying undirected graph.  For an acyclic orientation of $P_n$, denoted $G_{\bar {a}},$ we say that $i$ and $j$ are \textit{consecutive sinks} of $G_{\bar{a}}$ if $i$ and $j$ are both sinks of $G_{\bar{a}},$ and there are no other sinks in the circular interval $[i, j].$  Notice that this includes the sink with the largest label and the sink with the smallest label. 

\begin{prop}
Let $X_{\overrightarrow{P_n}}({\bf x},t) = \displaystyle \sum_{\lambda \vdash n} c_{\lambda} e_{\lambda}.$  Then $$c_{\lambda} = \displaystyle \sum_{G_{\bar{a}} \in AO_{\lambda}(P_n)} t^{\asc_{\overrightarrow{P_n}}(G_{\bar{a}})},$$ where $AO_{\lambda}(P_n)$ is the set of all acyclic orientations of $P_n$ such that the number of vertices between consecutive sinks is $\lambda_1 - 1,$ $\lambda_2 - 1,$...,$\lambda_k - 1$ in any order and $\asc_{\overrightarrow{P_n}}(G_{\bar{a}})$ is the number of directed edges of $\overrightarrow{P_n}$ that are oriented as in $G_{\bar{a}}.$
\end{prop}

\begin{example}
In the acyclic orientation of $P_8$ shown below, there are 3 vertices between sinks {\bf 2} and {\bf 6}, 1 vertex between sinks {\bf 6} and {\bf 8} and 1 vertex between sinks {\bf 8} and {\bf 2}, so this corresponds to $e_{422}.$  There are 4 edges that match the original orientation of $\overrightarrow{P_8},$ shown by the dotted red arrows, hence this acyclic orientation corresponds to $t^4 e_{422}.$

{\begin{center}
\includegraphics[scale=0.5]{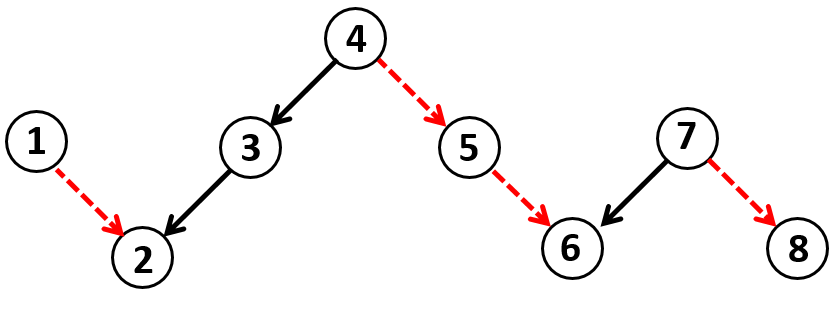}
\end{center}}
\end{example}

\begin{proof}
In \cite[Theorem 7.2]{Eul}, Shareshian and Wachs showed that $$\displaystyle \sum_{n \geq 0} X_{P_n}({\bf x},t)z^n = \frac{\displaystyle \sum_{k \geq 0}e_k z^k}{1-t \displaystyle \sum_{k \geq 2} [k-1]_t e_k z^k}.$$ 
From this we can get that for $n \geq 1$, \begin{equation} \label{Pn exp}
X_{\overrightarrow{P_n}}({\bf x},t) = \sum_{\lambda \vdash n}e_\lambda \sum_{\mu: \lambda(\mu) = \lambda}[\mu_1]_t t[\mu_2-1]_t t[\mu_3 -1]_t \cdots t[\mu_{l(\lambda)}-1]_t.
\end{equation}(See \cite[Table 1]{CQSF}).

We can obtain an acyclic orientation of $P_n$ from each term of the inner sum of \hyperref[Pn exp]{(\ref*{Pn exp})} as follows.  For each $2 \leq i \leq l(\lambda),$ suppose we choose $t^{j_i}$ from the $t[\mu_i -1]_t$ factor.  From this choice of $j_i$'s, we can create a sequence of mountains, $\overrightarrow{M}_{\mu_2+1, j_2}, \cdots \overrightarrow{M}_{\mu_{l(\lambda)}+1, j_{l(\lambda)}},$ with disjoint vertex sets.  Then we attach the mountains by identifying the last vertex of $\overrightarrow{M}_{\mu_i+1, j_i}$ with the first vertex of $\overrightarrow{M}_{\mu_{i+1}+1, j_{i+1}}$ for $2 \leq i <l(\lambda).$  

Now suppose we choose the $t^j$ term from the $[\mu_1]_t$ factor.  Then let $\overrightarrow{Q}_1 = (V, E)$ denote the digraph with vertex set $\{v_1, v_2, \cdots, v_{j+1}\}$ and edge set $E = \{(v_i, v_{i+1}) \mid 1 \leq i \leq j\}.$   We will say $v_1$ is the first vertex of $\overrightarrow{Q}_1$ and $v_{j+1}$ is the last.  For each $i = 1, 2, \cdots, j,$ we say that $v_{i+1}$ is the successor of $v_i$ and $v_i$ is the predecessor of $v_{i+1}.$  Let $\overrightarrow{Q}_2 = (V, E)$ denote the digraph with vertex set $\{v_1, v_2, \cdots, v_{\mu_1-j}\}$ and edge set $E = \{(v_i, v_{i-1}) \mid 1 < i \leq \mu_1-j\}.$  We will say $v_{1}$ is the first vertex of $\overrightarrow{Q}_2$ and $v_{\mu_1 - j}$ is the last.  For each $i = 1, 2, \cdots, \mu_1-j-1,$ we say that $v_{i+1}$ is the successor of $v_i$ and $v_i$ is the predecessor of $v_{i+1}.$

Then identify the last vertex of $\overrightarrow{Q}_1$ with the first vertex of $\overrightarrow{M}_{\mu_2+1, j_2}$ and identify the first vertex of $\overrightarrow{Q}_2$ with the last vertex of $\overrightarrow{M}_{\mu_{l(\lambda)+1}, j_{l(\lambda)}}.$

Label the resulting digraph by placing {\bf 1} on the first vertex, $v,$ of $\overrightarrow{Q}_1.$  Label the successor of $v$ with {\bf 2}, and continue labeling successors in order until all vertices are labeled.

It should be clear that we get a unique acyclic orientation in this manner and that every acyclic orientation can be built with this method.  This proves our proposition.
\end{proof}

\appendix
\section{Graph classes} \label{Graphs}

In this section, we will discuss a few properties of circular indifference digraphs.  Then we will take a look at how these graphs relate to other graphs found in the literature. 

\begin{definition}
Suppose we have a finite collection of arcs positioned around a circle of any radius so that no arc properly contains another.  We consider the starting point of an arc as the counterclockwise-most endpoint of the arc.  We can construct a digraph, which we call a \textit{proper circular arc digraph} by assigning a vertex to each arc and having an edge from arc $A$ to arc $B$ if the starting point of arc $B$ is contained in arc $A$.  The underlying undirected graph is called a \textit{proper circular arc graph}.
\end{definition}

\begin{example} Here we see a collection of proper circular arcs positioned around a circle and the corresponding proper circular arc digraph. 

\includegraphics[scale = 0.5]{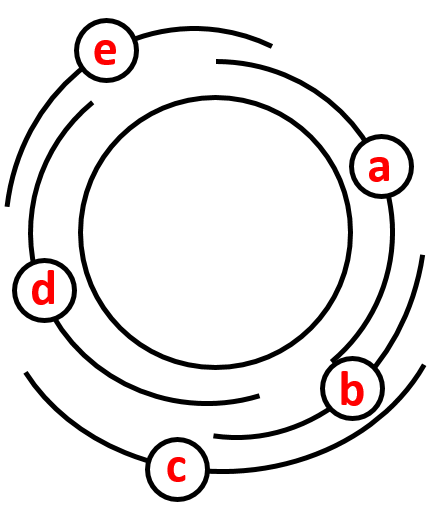}\ \ \ \ \ \ \ \  \ \ 
\includegraphics[scale = 0.5]{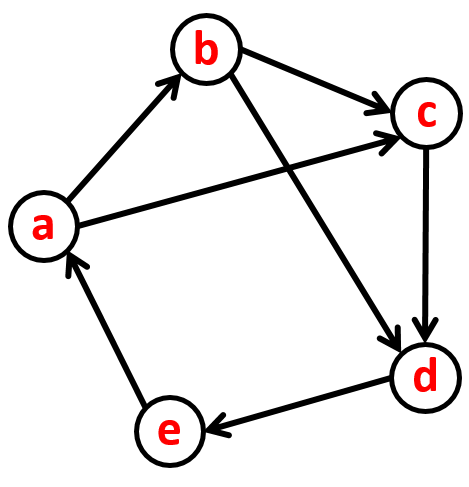}\centering 
\end{example}

\begin{thm} \label{circ digraphs}
Let $\overrightarrow{G}$ be a connected digraph.  Then the following statements are equivalent:
\begin{enumerate}
\item $\overrightarrow{G}$ is a proper circular arc digraph.
\item $\overrightarrow{G}$ is a circular indifference digraph.
\end{enumerate}

\end{thm}

\begin{proof}
First let's show (2)$\implies$(1).  Let $\overrightarrow{G}$ be a circular indifference digraph on $[n]$ that comes from the set of circular intervals $$I = \{[a_1, b_1], [a_2,b_2],\cdots,[a_k,b_k]\}$$ of $[n].$  From these intervals, we can construct a set of intervals $$\tilde{I} = \{[1, c_1], [2,c_2],...,[n,c_n]\}$$ such that for each $i \in [n],$ we have that $[i,c_i]$ is the largest circular interval that is contained in an interval of $I$ and that has $i$ as its left endpoint.   It is easy to see that each interval in $I$ must be contained in an interval of $\tilde{I}$ and vice versa, so $I$ and $\tilde{I}$ are both associated to $\overrightarrow{G}$.

We will construct $n$ proper arcs on a circle so that the corresponding proper circular arc graph is $\overrightarrow{G}$.  Draw a circle and place $n$ points equally spaced around the circle.  Label these points in cyclic order with $[n]$. For each circular interval of $\tilde{I},$ we will place an arc on the circle.  Start with a circular interval $[i, c_i] \in \tilde{I}$ of maximal size and draw an arc from slightly before $i$ to slightly after $c_i.$ Continue this process with all the circular intervals of $\tilde{I}$ in weakly decreasing order of their sizes.  Note that if the next interval has right endpoint the same $c_i$ as a previous interval, the newest arc (coming from a circular interval of smaller size) should extend slightly past the previous arc to avoid having one interval properly contained in another.

The arc we construct from the interval $[i, c_i]$ that starts just before $i$ on the circle is the arc that corresponds to vertex $i$ in $\overrightarrow{G}.$  Since arc $i$ will contain the starting points of all the arcs corresponding to the vertices in $[i, c_i],$ we have that $(i,j)$ is an edge of the proper circular arc digraph for each $i<j\leq c_i.$  From here, we can see that the proper circular arc digraph associated to this set of arcs is isomorphic to the circular indifference digraph, $\overrightarrow{G},$ we started with.

Here is an example of this process.  Suppose we have the circular intervals $$I = \{[1,3], [3,4], [4,5], [5,1]\}$$ on $[n].$  Then $$\tilde{I} = \{[1,3], [2,3], [3,4], [4,5], [5,1]\}.$$  The arcs that we would draw are shown in the figure below.  We can see that both $I$ and the arcs on this circle are associated with the digraph given below.

\begin{center} \includegraphics[scale=0.3]{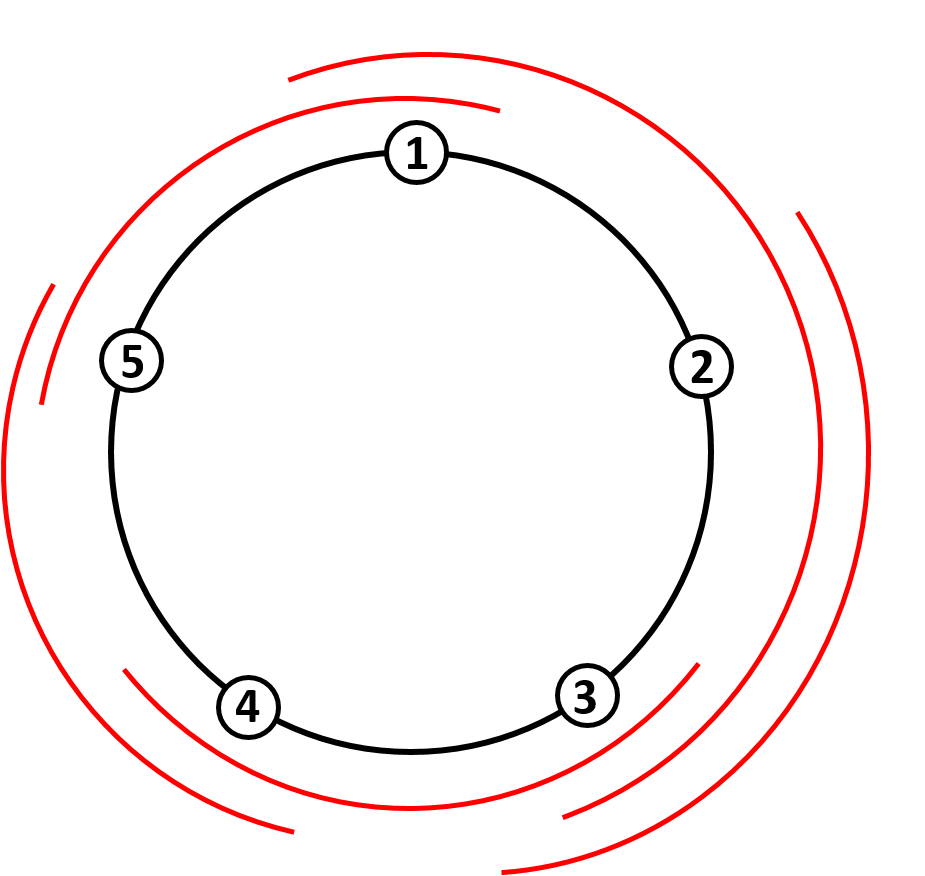} \hspace{0.5 in}
\includegraphics[scale=0.4]{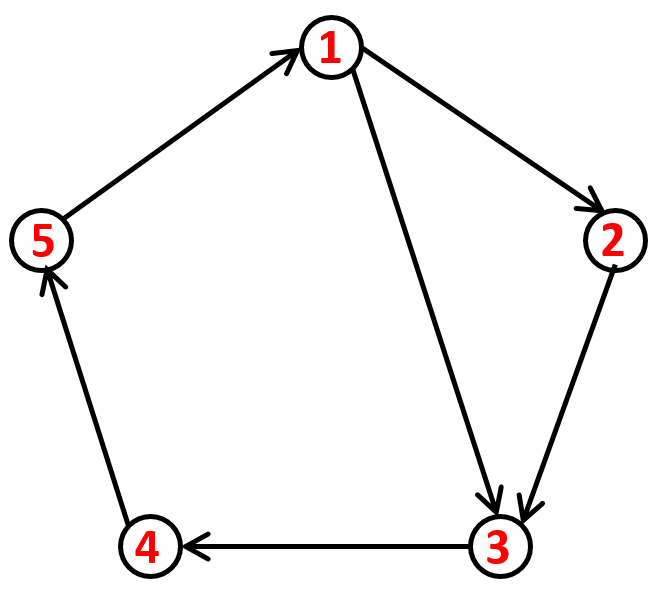}
\end{center}

Now let's show (1)$\implies$(2).

Let $\overrightarrow{G}$ be a proper circular arc digraph on $n$ vertices that comes from some proper arc representation on a circle.  Label one of the arcs $1$.  Now find the first arc that begins clockwise after arc $1$.  Label this arc $2$.  Then find the next arc that begins clockwise after arc $2$.  Label this arc $3$.  Continue this until all $n$ arcs are labeled with the labels $[n].$  Now create a set of circular intervals of $[n],$ called $I,$ as follows.  

Let $I$ be the set of all $[i,j]$ such that arc $i$ contains the starting point of arc $j.$  If arc $i$ contains the starting point of arc $j,$ then it must also contain the starting point of all arcs in $[i,j],$ hence we can see that the proper circular arc digraph associated with the set of arcs on the circle is isomorphic to the circular indifference digraph on $[n]$ associated with $I.$
\end{proof}

Another class of graphs we want to look at is the class of simple digraphs that do not have any induced subgraphs isomorphic to $\overrightarrow{K_{12}}$ and $\overrightarrow{K_{21}}$ as defined in \hyperref[Symmetry]{Section \ref*{Symmetry}}.  For notational convenience, we will call these $\{\overrightarrow{K_{12}}, \overrightarrow{K_{21}}\}$-free digraphs.
%

\begin{thm} \label{cyclic graphs}
Let $G$ be a simple connected graph. Then the following statements are equivalent:
\begin{enumerate} \label{circ graphs}
\item $G$ is isomorphic to a proper circular arc graph.
\item $G$ is isomorphic to a circular indifference graph.
\item $G$ admits an orientation that makes it a $\{\overrightarrow{K_{12}}, \overrightarrow{K_{21}}\}$-free digraph.
\end{enumerate}
\end{thm}

\begin{proof}
The equivalence of (1) and (3) was shown in \cite{Skrien}, and the equivalence of (1) and (2) follows from \hyperref[circ digraphs]{Theorem \ref*{circ digraphs}}.
\end{proof}
Now let us look at the non-circular version of these graphs.

\begin{definition}
Suppose that we have a collection of intervals, $I,$ of the ordered set $[n].$  Then we can construct a graph, $G = ([n], E),$ with edge set $E = \{\{i,j\} \mid i \neq j \mbox{ and } i, j \mbox{ contained in the same interval of }I \}$.  This is called an \textit{indifference graph}.
\end{definition}

\begin{definition}
Suppose we have a finite collection of intervals on the real line. We can associate a graph to this interval representation by letting each interval correspond to a vertex and allowing two distinct vertices to be adjacent if their corresponding intervals overlap.  This is called an \textit{interval graph}.  If no interval properly contains another, this is called a \textit{proper interval graph}.  If each interval has length 1, this is called a \textit{unit interval graph}.
\end{definition}

The following theorem is the acyclic or non-circular version of \hyperref[circ graphs]{Theorem \ref*{circ graphs}}.

\begin{thm}
Let $G$ be a simple graph. Then the following statements are equivalent:
\begin{enumerate}
\item $G$ is isomorphic to a proper interval graph.
\item $G$ is isomorphic to a unit interval graph.
\item $G$ is isomorphic to an indifference graph.
\item $G$ admits an acyclic orientation that makes it a $\{\overrightarrow{K_{12}}, \overrightarrow{K_{21}}\}$-free digraph.
\end{enumerate}
\end{thm}

The equivalence of $(1)$ and $(2)$ was shown by Roberts in \cite{Roberts}.  The equivalence of $(1)$ and $(4)$ was shown by Skrien in \cite{Skrien}. The equivalence of $(1)$ and $(3)$ is well-known, but can be shown by an analogous argument to the one given in the proof of Theorem 3.5.  

Note that if we turn a natural unit interval graph of Shareshian and Wachs into a digraph by orienting edges from smaller labels to larger labels, then we get an acyclic $\{\overrightarrow{K_{12}}, \overrightarrow{K_{21}}\}$-free digraph, and in fact, every acyclic $\{\overrightarrow{K_{12}}, \overrightarrow{K_{21}}\}$-free digraph comes from a natural unit interval graph.

\section*{Acknowledgements}{I would like to thank Richard Stanley for his helpful suggestion that made this work possible.  I would also like to thank my advisor, Michelle Wachs, for all of her guidance and encouragement.}

\bibliography{CQFDG}
\bibliographystyle{alpha}

%
%
%
%
%
%
%
%

\end{document}